\documentclass[reqno]{amsart}

\usepackage{amsmath,amssymb,amsthm,amsfonts,enumerate, enumitem,hyperref,cleveref}
\usepackage{dsfont}
\usepackage[all]{xy}
\usepackage[T1]{fontenc}
\usepackage{tikz}
\usepackage{pgfplots}
\pgfplotsset{compat=1.15}
\usepackage{mathrsfs}
\usetikzlibrary{arrows}

\DeclareMathOperator{\Aut}{Aut}
\DeclareMathOperator{\inn}{Inn}
\DeclareMathOperator{\laut}{LAut}
\DeclareMathOperator{\Span}{Span}
\DeclareMathOperator{\supp}{supp}
\DeclareMathOperator{\Max}{Max}
\DeclareMathOperator{\Min}{Min}

\theoremstyle{plain}
\newtheorem{theorem}{Theorem}[section]
\newtheorem{corollary}[theorem]{Corollary}
\newtheorem{proposition}[theorem]{Proposition}
\newtheorem{lemma}[theorem]{Lemma}

\newtheorem*{question}{Question}

\theoremstyle{definition}
\newtheorem{definition}[theorem]{Definition}
\newtheorem{remark}[theorem]{Remark}
\newtheorem{example}[theorem]{Example}

\crefrangeformat{equation}{#3(#1)#4--#5(#2)#6}

\crefname{theorem}{Theorem}{Theorems}
\crefname{lemma}{Lemma}{Lemmas}
\crefname{corollary}{Corollary}{Corollaries}
\crefname{proposition}{Proposition}{Propositions}
\crefname{definition}{Definition}{Definitions}
\crefname{example}{Example}{Examples}
\crefname{remark}{Remark}{Remarks}
\crefname{conjecture}{Conjecture}{Conjectures}
\crefname{section}{Section}{Sections}
\crefname{equation}{\unskip}{\unskip}
\crefname{enumi}{\unskip}{\unskip}
\crefname{subsection}{Subsection}{Subsections}

\newcommand{\0}{\theta}

\newcommand{\G}{\Gamma}
\newcommand{\af}{\alpha}
\newcommand{\bt}{\beta}
\newcommand{\lb}{\lambda}

\newcommand{\vf}{\varphi}

\newcommand{\dl}{\delta}

\newcommand{\sg}{\sigma}

\newcommand{\C}{\mathcal{C}}
\newcommand{\fC}{\mathfrak{C}}
\newcommand{\fD}{\mathfrak{D}}
\newcommand{\M}{\mathcal{M}}
\newcommand{\AM}{\mathcal{AM}}
\newcommand{\cP}{\mathcal{P}}
\newcommand{\Cr}{\mathrm{Cr}}
\newcommand{\K}{\mathrm{K}}

\newcommand{\m}{{}^{-1}}
\newcommand{\sst}{\subseteq}
\newcommand{\impl}{\Rightarrow}

\newcommand{\lf}{\lfloor}
\newcommand{\rf}{\rfloor}
\newcommand{\wtl}{\widetilde}
\newcommand{\ch}{\mathrm{char}}

\newcommand{\id}{\mathrm{id}}

\begin{document}
	\title[Proper Lie automorphisms of incidence algebras]{Proper Lie automorphisms of incidence algebras}	
	\author{\'Erica Z. Fornaroli}
	\address{Departamento de Matem\'atica, Universidade Estadual de Maring\'a, Maring\'a, PR, CEP: 87020--900, Brazil}
	\email{ezancanella@uem.br}
	
	\author{Mykola Khrypchenko}
	\address{Departamento de Matem\'atica, Universidade Federal de Santa Catarina,  Campus Reitor Jo\~ao David Ferreira Lima, Florian\'opolis, SC, CEP: 88040--900, Brazil}
	\email{nskhripchenko@gmail.com}
	
	\author{Ednei A. Santulo Jr.}
	\address{Departamento de Matem\'atica, Universidade Estadual de Maring\'a, Maring\'a, PR, CEP: 87020--900, Brazil}
	\email{easjunior@uem.br}
	
	
	\begin{abstract}
		Let $X$ be a finite connected poset and $K$ a field. We study the question, when all Lie automorphisms of the incidence algebra $I(X,K)$ are proper. Without any restriction on the length of $X$ we find only a sufficient condition involving certain equivalence relation on the set of maximal chains of $X$. For some classes of posets of length one, such as finite connected crownless posets (i.e., without weak crown subposets), crowns and ordinal sums of two antichains we give a complete answer.
	\end{abstract}
	
	\maketitle
	
	\tableofcontents
	
	\section*{Introduction}
	
	A {\it Lie isomorphism} of associative rings $(R,\cdot)$ and $(S,\cdot)$ is an isomorphism of the corresponding Lie rings $\left(R,[\phantom{a},\phantom{a}]\right)$ and $\left(S,[\phantom{a},\phantom{a}]\right)$, where $[a,b]=a\cdot b-b\cdot a$. If $R=S$, a Lie isomorphism $R\to S$ is called a {\it Lie automorphism} of $R$. If, moreover, $R$ and $S$ are algebras, it is natural to require Lie isomorphisms $R\to S$ to be linear. Any bijective map of the form $\phi+\nu$, where $\phi$ is either an isomorphism $R\to S$ or the negative of an anti-isomorphism $R\to S$, and $\nu$ is an additive map on $R$ with values in the center of $S$ whose kernel contains $[R,R]$, is always a Lie isomorphism. Such Lie isomorphisms are called {\it proper}. In most of the cases studied in the literature, these are the only examples of Lie isomorphisms. Indeed, this is true for Lie automorphisms of full matrix rings $M_n(R)$ over division rings $R$ with $\ch(R)\not\in\{2,3\}$ as proved in~\cite{Hua51}, for Lie isomorphisms of primitive~\cite{Martindale63}, simple~\cite{Martindale69simple} and prime rings~\cite{Martindale69}, for Lie automorphisms of upper triangular matrix algebras $T_n(R)$ over commutative rings~\cite{dzD,Cao97} and for Lie isomorphisms of block-triangular matrix algebras over a UFD~\cite{Cecil}.
	
	In~\cite{FKS} we described Lie automorphisms of the incidence algebra $I(X,K)$ of a finite connected poset $X$ over a field $K$. In general, they are not proper as shown in~\cite[Example 5.20]{FKS}, but, for some classes of posets, every Lie automorphism of $I(X,K)$ is proper. For instance, if $X$ is a chain of cardinality $n$, then $I(X,K)\cong T_n(K)$, and thus every Lie automorphism of $I(X,K)$ is proper in view of \cite[Corollary~5.19]{FKS} (see also~\cite[Theorem~6]{dzD}). So the following question arises.
	\begin{question}
		What are the necessary and sufficient conditions on a finite connected poset $X$ such that all Lie automorphisms of $I(X,K)$ are proper?
	\end{question}
	In this paper we give a partial answer to this question. Namely, for a general $X$ we find only a sufficient condition (see \cref{|C(X)-over-tilde|=1}), and for some particular classes of posets of length one $X$ we give a complete answer (see~\cref{P(X)=AM(X)-for-X-without-crown,P(X)-and-AM(X)-for-X=2-crown,P(X)=AM(X)-for-X=K_mn}). 
	
	More precisely, our work is organized as follows. \cref{sec-prelim} serves as a background on posets, incidence algebras and maps on them. In particular, we recall all the necessary definitions from~\cite{FKS} and introduce some new notations. In \cref{sec-proper} we reduce the question of when all Lie automorphisms of $I(X,K)$ are proper to a purely combinatorial property of $X$ dealing with a certain group $\AM(X)$ of bijections on maximal chains of $X$ (see \cref{all-proper-iff-cP(X)=AM(X)}). In our terminology, we prove that every Lie automorphism of $I(X,K)$ is proper if and only if every bijection $\0\in\AM(X)$ is {\it proper}. In \cref{sec-max-chains} we introduce an equivalence relation $\sim$ on maximal chains of $X$ and show that any $\0\in\AM(X)$ induces isomorphisms or anti-isomorphisms between certain subsets of $X$ (the so-called {\it supports} of $\sim$-classes), as proved in \cref{supp(C)-and-supp(0(C))}. Consequently, if all the maximal chains of $X$ are equivalent, then all Lie automorphisms of $I(X,K)$ are proper (see \cref{|C(X)-over-tilde|=1}). The equivalence $\sim$ is just the equality relation whenever $X$ is of length one, hence this situation is treated separately. This is done in \cref{sec-length-one}. We first consider the case when $X$ has no crown subset and give a full description of those $X$ for which any $\0\in\AM(X)$ is proper (see \cref{P(X)=AM(X)-for-X-without-crown}). We then pass to two specific classes of $X$: $n$-crowns $\mathrm{Cr}_n$ and ordinal sums of two anti-chains $\mathrm{K}_{m,n}$. If $X=\mathrm{Cr}_n$, then we explicitly describe the group $\AM(X)$ (see \cref{AM(C_n)-as-semidirect-product}) and its subgroup of proper bijections (see \cref{P(C_n)-as-semidirect-product}). It follows that all $\0\in\AM(\mathrm{Cr}_n)$ are proper exactly when $n=2$ (see \cref{P(X)-and-AM(X)-for-X=2-crown}). If $X=\mathrm{K}_{m,n}$, then there are only proper $\0\in\AM(X)$ as proved in \cref{P(X)=AM(X)-for-X=K_mn}.
	
	\section{Preliminaries}\label{sec-prelim}
	
	\subsection{Automorphisms and anti-automorphisms}
	
	Let $A$ be an algebra. We denote by $\Aut(A)$ the group of (linear) automorphisms of $A$, by $\Aut^-(A)$ the set of (linear) anti-automorphisms of $A$ and by $\Aut^\pm(A)$ the union $\Aut(A)\cup\Aut^-(A)$. Observe that the union is non-disjoint if and only if $A$ is commutative, in which case $\Aut(A)=\Aut^-(A)$. The set $\Aut^\pm(A)$ is a group under the composition, and moreover, if $A$ is non-commutative and $\Aut^-(A)\ne\emptyset$, then $\Aut(A)$ is a (normal) subgroup of $\Aut^\pm(A)$ of index $2$. In particular, $|\Aut(A)|=|\Aut^-(A)|$, whenever $A$ is non-commutative and $\Aut^-(A)\ne\emptyset$. We use the analogous notations $\Aut(X)$, $\Aut^-(X)$ and $\Aut^\pm(X)$ for the group of automorphisms of a poset $X$, the set of anti-automorphisms of $X$ and the group $\Aut(X)\cup\Aut^-(X)$, respectively. As above, $\Aut(X)$ either coincides with $\Aut^\pm(X)$ or is a subgroup of index $2$ in $\Aut^\pm(X)$ (if $X$ is not an anti-chain and $\Aut^-(X)\ne\emptyset$).
	
	\subsection{Posets}
	
	Let $(X,\le)$ be a partially ordered set (which we usually shorten to ``poset'') and $x,y\in X$. The \emph{interval} from $x$ to $y$ is the set $\lfloor x,y \rfloor=\{z\in X : x\leq z\leq y\}$. The poset $X$ is said to be \textit{locally finite} if all the intervals of $X$ are finite. A \textit{chain} in $X$ is a linearly ordered (under the induced order) subset of $X$. The \textit{length} of a finite chain $C\sst X$ is defined to be $|C|-1$. The \textit{length}\footnote{often also called the height} of a finite poset $X$, denoted by $l(X)$, is the maximum length of chains $C\sst X$. A {\it walk} in $X$ is a sequence $x_0,x_1,\dots,x_m\in X$, such that $x_i$ and $x_{i+1}$ are comparable and $l(\lf x_i,x_{i+1}\rf)=1$ (if $x_i\le x_{i+1}$) or $l(\lf x_{i+1},x_i\rf)=1$ (if $x_{i+1}\le x_i$) for all $i=0,\dots,m-1$. A walk $x_0,x_1,\dots,x_m$ is {\it closed} if $x_0=x_m$. A {\it path} is a walk satisfying $x_i\ne x_j$ for $i\ne j$. A {\it cycle} is a closed walk $x_0,x_1,\dots,x_m=x_0$ in which $m\ge 4$ and $x_i=x_j\impl \{i,j\}=\{0,m\}$  for $i\ne j$. We say that $X$ is {\it connected} if for any pair of $x,y\in X$ there is a path $x=x_0,\dots,x_m=y$. We will denote by $\Min(X)$ (resp. $\Max(X)$) the set of minimal (resp. maximal) elements of $X$. If $X$ is connected and $|X|>1$, then $\Min(X)\cap\Max(X)=\emptyset$.
	
	\subsection{Incidence algebras}
	
	Let $X$ be a locally finite poset and $K$ a field. The \emph{incidence algebra}~\cite{Rota64} $I(X,K)$ of $X$ over $K$ is the $K$-space of functions $f:X\times X\to K$ such that $f(x,y)=0$ if $x\nleq y$.  This is a unital $K$-algebra under the convolution product
	$$
	(fg)(x,y)=\sum_{x\leq z\leq y}f(x,z)g(z,y),
	$$
	for any $f, g\in I(X,K)$. Its identity element $\delta$ is given by
	\begin{align*}
	\delta(x,y)=
	\begin{cases}
	1, & x=y,\\
	0, & x\ne y.
	\end{cases}
	\end{align*}
	
	Throughout the rest of the paper $X$ will stand for a connected finite poset. Then $I(X,K)$ admits the standard basis $\{e_{xy} : x\leq y\}$, where
	\begin{align*}	
	e_{xy}(u,v)=
	\begin{cases}
	1, & (u,v)=(x,y),\\
	0, & (u,v)\ne(x,y).
	\end{cases}
	\end{align*}
	We will write $e_{x}=e_{xx}$. Denote also $B=\{e_{xy} : x<y\}$. It is a well-known fact (see~\cite[Theorem~4.2.5]{SpDo}) that the Jacobson radical of $I(X,K)$ is
	$$
	J(I(X,K))=\{f\in I(X,K) : f(x,x)=0 \text{ for all } x\in X\}=\Span_K B.
	$$
	{\it Diagonal elements} of $I(X,K)$ are those $f\in I(X,K)$ satisfying  $f(x,y)=0$ for $x\neq y$. They form a commutative subalgebra $D(X,K)$ of $I(X,K)$ spanned by $\{e_{x} : x \in X\}$. Clearly, each $f\in I(X,K)$ can be uniquely written as $f=f_D+f_J$ with $f_D\in D(X,K)$ and $f_J\in J(I(X,K))$.
	
	\subsection{Decomposition of $\phi\in\Aut^\pm(I(X,K))$} 
	
	Now, we recall the descriptions of automorphisms and anti-automorphisms of $I(X,K)$. Firstly, if $X$ and $Y$ are finite posets and $\lambda: X\to Y$ is an isomorphism (resp. anti-isomorphism), then $\lambda$ \emph{induces} an isomorphism (resp. anti-isomorphism) $\hat\lambda: I(X,K)\to I(Y,K)$ defined by $\hat\lambda(e_{xy})=e_{\lambda(x)\lambda(y)}$ (resp. $\hat\lambda(e_{xy})=e_{\lambda(y)\lambda(x)}$), for all $x\leq y$ in $X$. 
	%If $\lambda\in \Aut(X)$ (resp. $\Aut^-(X)$), then $\lambda$ \emph{induces} an automorphism (resp. anti-automorphism) of $I(X,K)$ defined by $\hat\lambda(e_{xy})=e_{\lambda(x)\lambda(y)}$ (resp. $\hat\lambda(e_{xy})=e_{\lambda(y)\lambda(x)}$), for all $x\leq y$ in $X$. %We will call $\hat\lambda$ the \emph{automorphism} (resp. \emph{anti-automorphism}) induced by $\lambda$. 
	An element $\sigma\in I(X,K)$ such that $\sigma(x,y)\neq 0$, for all $x\leq y$, and $\sigma(x,y)\sigma(y,z)=\sigma(x,z)$ whenever $x\leq y\leq z$, determines an automorphism $M_{\sigma}$ of $I(X,K)$ by $M_{\sigma}(e_{xy})=\sigma(x,y)e_{xy}$, for all $x\leq y$. Such automorphisms are called \emph{multiplicative}. Any automorphism (anti-automorphism) of $I(X,K)$ decomposes as
	%If $\phi$ is an automorphism (anti-automorphism) of $I(X,K)$, then 
	\begin{align}\label{decomposition}
	\phi=\hat\lambda \circ \xi \circ M_{\sigma},
	\end{align} 
	where $\lambda\in\Aut(X)$ (resp. $\Aut^-(X)$), $\xi $ is an inner automorphism and $M_{\sigma}$ is a multiplicative automorphism of $I(X,K)$. For automorphisms this was proved in \cite[Theorem~5]{Baclawski72} and for anti-automorphisms in \cite[Theorem~5]{BFS12} (for more results on automorphisms and anti-automorphisms of incidence algebras see~\cite{BruFS,BruL,BFS12,BFS15,St,Sp,Drozd-Kolesnik07,Kh-aut}).
	
	\subsection{Lie automorphisms of incidence algebras}
	
	In this section we introduce several new notations and recall some definitions and results from \cite{FKS}. 
	
	We denote by $\C(X)$ the set of maximal chains in $X$. Let $C: u_1<u_2<\dots<u_m$ in $\C(X)$. A bijection $\0:B\to B$ is \textit{increasing (resp. decreasing) on $C$} if there exists $D: v_1<v_2<\dots<v_m$ in $\C(X)$ such that $\0(e_{u_iu_j})=e_{v_iv_j}$ for all $1\le i<j\le m$ (resp. $\0(e_{u_iu_j})=e_{v_{m-j+1}v_{m-i+1}}$ for all $1\le i<j\le m$). In this case we write $\0(C)=D$. Moreover, we say that $\0$ \textit{is monotone on maximal chains in $X$} if, for any $C\in \C(X)$, $\0$ is increasing or decreasing on $C$. We denote by $\M(X)$ the set of bijections $B\to B$ which are monotone on maximal chains in $X$. It is easy to see that $\M(X)$ is a subgroup of the symmetric group $S(B)$. Each $\0\in\M(X)$ induces a bijection on $\C(X)$ which maps $C$ to $\0(C)$.
	
	Let $\0:B\to B$ be a bijection and $X^2_<=\{(x,y)\in X^2: x<y\}$. A map $\sg:X^2_<\to K^*$ is \textit{compatible} with $\0$ if $\sg(x,z)=\sg(x,y)\sg(y,z)$ whenever $\0(e_{xz})=\0(e_{xy})\0(e_{yz})$, and $\sg(x,z)=-\sg(x,y)\sg(y,z)$ whenever $\0(e_{xz})=\0(e_{yz})\0(e_{xy})$. 
	
	Let $\0:B\to B$ be a bijection and $\Gamma: u_0,u_1,\dots,u_m=u_0$ a closed walk in $X$. In~\cite{FKS} we introduced the following $4$ functions $X\to\mathds{N}$:
	\begin{align*}
	s^+_{\0,\G}(z)&=|\{i: u_i<u_{i+1}\text{ and }\exists w>z\text{ such that }\0(e_{zw})=e_{u_iu_{i+1}}\}|,\\	
	s^-_{\0,\G}(z)&=|\{i: u_i>u_{i+1}\text{ and }\exists w>z\text{ such that }\0(e_{zw})=e_{u_{i+1}u_i}\}|,\\
	t^+_{\0,\G}(z)&=|\{i: u_i<u_{i+1}\text{ and }\exists w<z\text{ such that }\0(e_{wz})=e_{u_iu_{i+1}}\}|,\\
	t^-_{\0,\G}(z)&=|\{i: u_i>u_{i+1}\text{ and }\exists w<z\text{ such that }\0(e_{wz})=e_{u_{i+1}u_i}\}|.
	\end{align*}
	We call the bijection $\0:B\to B$ \textit{admissible} if 
	\begin{align}\label{s^+-s^-=t^+-t^-}
	s^+_{\0,\G}(z)-s^-_{\0,\G}(z)=t^+_{\0,\G}(z)-t^-_{\0,\G}(z)
	\end{align}
	for any closed walk $\G:u_0,u_1,\dots,u_m=u_0$ in $X$ and for all $z\in X$. In particular, if $X$ is a tree, then any bijection $\0:B\to B$ is admissible. We denote by $\AM(X)$ the set of those $\0\in\M(X)$ which are admissible. 
	
	Let $X=\{x_1,\dots, x_n\}$. Given $\0\in \AM(X)$, a map $\sg:X_<^2\to K^*$ compatible with $\0$ and a sequence $c=(c_1,\dots,c_n)\in K^n$ such that $\sum_{i=1}^nc_i\in K^*$, we define in \cite[Definition~5.17]{FKS} the following \emph{elementary Lie automorphism} $\tau=\tau_{\0,\sg,c}$ of $I(X,K)$ where, for any $e_{xy}\in B$, 
	$$\tau(e_{xy})=\sigma(x,y)\0(e_{xy})$$
	and $\tau|_{D(X,K)}$ is determined by
	$$\tau(e_{x_i})(x_1,x_1)=c_i,$$
	$i=1,\dots,n$, as in Lemmas~5.8 and 5.16 from \cite{FKS}. As in \cite[Definition~5.15]{FKS} we say that $\tau$ \emph{induces} the pair $(\0,\sg)$ and in some situations we write $\0=\0_{\tau}$.
	
	As in \cite{FKS}, we denote by $\laut(I(X,K))$ the group of Lie automorphisms of $I(X,K)$ and by $\wtl\laut(I(X,K))$ its subgroup of elementary Lie automorphisms.  
	We will also use the notation $\inn_1(I(X,K))$ for the subgroup of inner automorphisms consisting of conjugations by $\bt\in I(X,K)$ with $\bt_D=\dl$.
	
	\begin{theorem}\cite[Theorem~4.15]{FKS}\label{semidireto}
		The group $\laut(I(X,K))$ is isomorphic to the semidirect product $\inn_1(I(X,K))\rtimes\wtl\laut(I(X,K))$.
	\end{theorem}
	
	\section{Proper Lie automorphisms of $I(X,K)$ and proper bijections of $B$}\label{sec-proper}
	
	Let $\varphi\in\laut(I(X,K))$. Then $\varphi=\psi\circ \tau_{\0,\sg,c}$, where $\psi\in \inn_1(I(X,K))$ and $\tau_{\0,\sg,c}$ is an elementary
	Lie automorphism of $I(X,K)$, by \cref{semidireto}. Note that $\varphi$ is proper if and only if $\tau_{\0,\sg,c}$ is proper. Therefore, all Lie automorphisms of $I(X,K)$ are proper if and only if all elementary Lie automorphisms of $I(X,K)$ are proper.
	
	Let $\varphi=\tau_{\0,\sg,c}$ be an elementary Lie automorphism of $I(X,K)$. Suppose that $\varphi$ is proper, $\varphi=\phi+\nu$, where $\phi\in\Aut(I(X,K))$ or $-\phi\in\Aut^-(I(X,K))$ and $\nu$ is a linear central-valued map on $I(X,K)$ such that $\nu([I(X,K),I(X,K)])=\{0\}$. If $x<y$, then $e_{xy}\in J(I(X,K))=[I(X,K),I(X,K)]$, by \cite[Proposition~2.3]{FKS}. Thus
	\begin{align}\label{vf<}
	\varphi(e_{xy})=\phi(e_{xy}), \forall x<y.
	\end{align}
	By \cite[Corollary~1.3.15]{SpDo}, for each $x\in X$ there is $\alpha_x\in K$ such that
	\begin{align}\label{vf=}
	\varphi(e_{x})=\phi(e_{x})+\alpha_x\delta.
	\end{align}
	
	Suppose firstly that $\phi\in\Aut(I(X,K))$. Then, by \cref{decomposition}, $\phi=\hat\lambda \circ \xi_f \circ M_{\tau}$, where $\lambda\in\Aut(X)$, $\xi_f$ is an inner automorphism and $M_{\tau}$ is a multiplicative automorphism of $I(X,K)$. Thus, by \cref{vf<},
	\begin{align}\label{theta_esp_aut}
	\0(e_{xy})= & \sg(x,y)^{-1}(\hat\lambda \circ \xi_f \circ M_{\tau})(e_{xy})=\sg(x,y)^{-1}(\hat\lambda \circ \xi_f)(\tau(x,y)e_{xy}) \nonumber\\
	= & \sg(x,y)^{-1}\hat\lambda(\tau(x,y)fe_{xy}f^{-1})=\sg(x,y)^{-1}\tau(x,y)\hat\lambda(f)e_{\lambda(x)\lambda(y)}\hat\lambda(f)^{-1}.
	\end{align}
	
	Analogously, if $-\phi\in\Aut^-(I(X,K))$, then, by \cref{decomposition}, $\phi=\hat\lambda \circ \xi_f \circ M_{\tau}$, where $\lambda\in\Aut^-(X)$, $\xi_f$ is an inner automorphism and $M_{\tau}$ is a multiplicative automorphism of $I(X,K)$. Thus, by \cref{vf<},
	\begin{align}\label{theta_esp_-anti-aut}
	\0(e_{xy})= \sg(x,y)^{-1}(\hat\lambda \circ \xi_f \circ M_{\tau})(e_{xy})= \sg(x,y)^{-1}\tau(x,y)\hat\lambda(f)^{-1}e_{\lambda(y)\lambda(x)}\hat\lambda(f).
	\end{align}
	
	\begin{remark}\label{e_xy_conjugado_e_uv}
		Let $x\leq y, u\leq v$ in $X$, $\alpha\in K^{\ast}$ and $h\in I(X,K)$ an invertible element. Note that if $he_{xy}h^{-1}=\alpha e_{uv}$, then $(x,y)=(u,v)$.
	\end{remark}
	
	\begin{definition}\label{propertheta}
		A bijection $\0:B\to B$ is said to be {\it proper} if there exists $\lambda\in \Aut^\pm(X)$ such that $\0(e_{xy})=\hat\lb(e_{xy})$ for all $e_{xy}\in B$. The proper bijections of $B$ form a group, which we denote by $\cP(X)$.
	\end{definition}
	
	\begin{proposition}\label{P(X)-cong-Aut^pm(X)}
		If $|X|>2$, then the group $\cP(X)$ is isomorphic to $\Aut^\pm(X)$.
	\end{proposition}
	\begin{proof} 
		The map sending $\lb\in\Aut(X)$ (resp. $\lb\in\Aut^-(X)$) to $\0\in\cP(X)$, such that $\0(e_{xy})=e_{\lb(x)\lb(y)}$ (resp. $\0(e_{xy})=e_{\lb(y)\lb(x)}$), is an epimorphism from $\Aut^\pm(X)$ to $\cP(X)$. We only need to prove that it is injective. 
		
		It is injective on $\Aut(X)$. Indeed, take $\lb,\mu\in\Aut(X)$ such that $(\lb(x),\lb(y))=(\mu(x),\mu(y))$ for all $x<y$ in $X$. Let $x$ be an arbitrary element of $X$. Since $X$ is connected and $|X|>1$, there is $y\in X$ such that either $y<x$ or $y>x$. In both cases we get $\lb(x)=\mu(x)$. Thus, $\lb=\mu$. Similarly, one proves injectivity on $\Aut^-(X)$.
		
		Assume now that there are $\lb\in\Aut(X)$ and $\mu\in\Aut^-(X)$ such that $(\lb(x),\lb(y))=(\mu(y),\mu(x))$ for all $x<y$ in $X$. We first show that $X$ must have length at most $1$. Indeed, if there were $x<y<z$ in $X$, then we would have $(\lb(x),\lb(y))=(\mu(y),\mu(x))$ and $(\lb(y),\lb(z))=(\mu(z),\mu(y))$, whence $\lb(x)=\lb(z)$, a contradiction. Now, consider a triple $x,y,z\in X$ with $x>z<y$. It follows from $(\lb(z),\lb(x))=(\mu(x),\mu(z))$ and $(\lb(z),\lb(y))=(\mu(y),\mu(z))$ that $\mu(x)=\mu(y)$, a contradiction. Similarly, the existence of a triple $x,y,z\in X$ with $x<z>y$ leads to $\lb(x)=\lb(y)$. If there are two incomparable elements $x,y\in X$, then there exists a sequence $x=x_1,\dots,x_m=y$, where $m\ge 3$ and either $x_1<x_2>x_3$ or $x_1>x_2<x_3$. In both cases we come to a contradiction. Thus, $X$ is of length at most $1$ and any two elements of $X$ are comparable, which means that $X$ is either a singleton, or a chain of length $1$.
	\end{proof}
	
	\begin{remark}
		If $X$ is a chain of length $1$, then $|\cP(X)|=1$, while $|\Aut^\pm(X)|=2$.
	\end{remark}
	
	\begin{proposition}
		We have $\cP(X)\sst\AM(X)$.
	\end{proposition}
	\begin{proof}
		Let $\0\in\cP(X)$. Then there exists $\lambda\in\Aut^\pm(X)$ as in \cref{propertheta}. In both cases $\0$ is the restriction of $\hat{\lambda}$ to $B$ and, since $\hat{\lambda}$ or $-\hat{\lambda}$ is an elementary Lie automorphism of $I(X,K)$ by \cite[Remark~4.8]{FKS}, then $\0\in\AM(X)$ by \cite[Remark~5.10]{FKS}. 
		%%Assume first that $\0(e_{xy})=e_{f(x)\lambda(y)}$ for some $f\in\Aut(X)$. Let $\Gamma: u_0,u_1,\dots,u_m=u_0$ be a cycle in $X$ and $z\in X$. There exists a unique cycle $\Delta: v_0,v_1,\dots,v_m=v_0$ in $X$ such that $f(v_i)=u_i$, $0\le i\le m$. More precisely, $u_i<u_{i+1}\iff v_i<v_{i+1}$, $0\le i\le m$. If $z\not\in\Delta$, then $s^+_{\0,\G}(z)=s^-_{\0,\G}(z)=t^+_{\0,\G}(z)=t^-_{\0,\G}(z)=0$. Let $z\in\Delta$, so that $z=v_i$ for a unique $v_i$. Assume, for simplicity, that $0<i<m$ (otherwise consider all the indices ``modulo $m$''). There are 4 cases. 
		%%If $u_{i-1}<u_i<u_{i+1}$, then $s^+_{\0,\G}(z)=t^+_{\0,\G}(z)=1$ and $s^-_{\0,\G}(z)=t^-_{\0,\G}(z)=0$.
		%%If $u_{i-1}<u_i>u_{i+1}$, then $t^+_{\0,\G}(z)=t^-_{\0,\G}(z)=1$ and $s^+_{\0,\G}(z)=s^-_{\0,\G}(z)=0$.
		%%If $u_{i-1}>u_i<u_{i+1}$, then $s^+_{\0,\G}(z)=s^-_{\0,\G}(z)=1$ and $t^+_{\0,\G}(z)=t^-_{\0,\G}(z)=0$.
		%%If $u_{i-1}>u_i>u_{i+1}$, then $s^-_{\0,\G}(z)=t^-_{\0,\G}(z)=1$ and $s^+_{\0,\G}(z)=t^+_{\0,\G}(z)=0$.
		%%In any case 
		%%\begin{align}\label{s^+-s^-=t^+-t^-}
		%%			s^+_{\0,\G}(z)-s^-_{\0,\G}(z)=t^+_{\0,\G}(z)-t^-_{\0,\G}(z),
		%%			\end{align}
		%%  so that $\0$ is admissible.
		%%If $\0(e_{xy})=e_{\lambda(y)f(x)}$ for some anti-automorphism $f$ of $X$, then in the 4 cases above we replace the superscript ``$+$'' by ``$-$'' and ``$-$'' by ``$+$''. Nevertheless, \cref{s^+-s^-=t^+-t^-} still holds. 
	\end{proof}
	
	\begin{lemma}\label{theta_proper_vf_proper}
		Let $\vf\in\wtl\laut(I(X,K))$ inducing a pair $(\0,\sg)$. Then $\0$ is proper if, and only if, $\vf$ is proper.
	\end{lemma}
	\begin{proof}
		Assume first that $\0(e_{xy})=e_{\lambda(x)\lambda(y)}$ for some $\lambda\in\Aut(X)$. Then $\0$ is increasing on any maximal chain in $X$ and thus $\sg(x,z)=\sg(x,y)\sg(y,z)$ for all $x<y<z$. Extending $\sg$ to $X^2_\le=\{(x,y) : x\le y\}$ by means of $\sg(x,x)=1$ for all $x\in X$, we obtain the multiplicative automorphism $M_\sg\in\Aut(I(X,K))$. Consider $\psi=\vf\circ M\m_\sg$. Notice that $\psi(e_{xy})=e_{\lambda(x)\lambda(y)}$ for all $x<y$ and $\psi(e_x)=\vf(e_x)$ for all $x\in X$. It suffices to prove that $\psi$ is proper. Indeed, if $\psi=\phi+\nu$, then $\vf=\psi\circ M_\sg=\phi\circ M_\sg+\nu$ since $M_\sg$ is identity on $D(X,R)$. 
		
		Clearly, $\psi(e_{xy})=\hat\lambda(e_{xy})$ for all $x<y$, so $\hat\lambda$ is a candidate for $\phi$. It remains to prove \cref{vf=} for $\psi$, i.e., to show that $\psi(e_x)=e_{\lambda(x)}+\af_x\dl$ for some $\af_x\in K$. The latter is equivalent to
		\begin{align}\label{psi(e_x)(y_y)=psi(e_x)(lambda(x)_lambda(x))-1}
		\psi(e_x)(y,y)=\psi(e_x)(\lambda(x),\lambda(x))-1
		\end{align}
		for all $y\ne \lambda(x)$. Given $x,y\in X$ with $y\ne \lambda(x)$, we choose a path $\lambda(x)=u_0,\dots,u_m=y$ from $\lambda(x)$ to $y$. Let $v_i\in X$ such that $\lb(v_i)=u_i$, $0\le i\le m$. In particular, $v_0=x$. Then
		\begin{align}\label{psi(e_x)(y_y)=psi(e_x)(lambda(x)_lambda(x))+sum}
		\psi(e_x)(y,y)=\psi(e_x)(\lambda(x),\lambda(x))+\sum_{i=0}^{m-1}(\psi(e_x)(u_{i+1},u_{i+1})-\psi(e_x)(u_i,u_i)).
		\end{align}
		If $u_0<u_1$, then $\0(e_{xv_1})=e_{u_0u_1}$, so $\psi(e_x)(u_1,u_1)-\psi(e_x)(u_0,u_0)=-1$ by~\cite[Lemma~5.7]{FKS}. Similarly, if $u_0>u_1$, then $\0(e_{v_1x})=e_{u_1u_0}$, so $\psi(e_x)(u_0,u_0)-\psi(e_x)(u_1,u_1)=1$. In any case $\psi(e_x)(u_1,u_1)-\psi(e_x)(u_0,u_0)=-1$. Observe that $v_i\ne x$ for all $i>0$. Hence $\psi(e_x)(u_{i+1},u_{i+1})-\psi(e_x)(u_i,u_i)=0$ for all such $i$ by~\cite[Lemma~5.7]{FKS}. It follows that the sum on the right-hand side of \cref{psi(e_x)(y_y)=psi(e_x)(lambda(x)_lambda(x))+sum} has only one non-zero term which equals $-1$, proving \cref{psi(e_x)(y_y)=psi(e_x)(lambda(x)_lambda(x))-1}.
		
		The case $\0(e_{xy})=e_{\lambda(y)\lambda(x)}$, where $\lambda\in\Aut^-(X)$, is similar.
		
		Conversely, suppose that $\vf=\phi+\nu$, where $\phi\in\Aut(I(X,K))$ (resp. $-\phi\in\Aut^-(I(X,K))$) and $\nu$ is a linear central-valued map on $I(X,K)$ annihilating $[I(X,K),I(X,K)]$. It follows from \cref{theta_esp_aut} (resp. \cref{theta_esp_-anti-aut}) and \cref{e_xy_conjugado_e_uv} that $\0(e_{xy})=e_{\lambda(x)\lambda(y)}$ ($\0(e_{xy})=e_{\lambda(y)\lambda(x)}$) for all $e_{xy}\in B$, where $\lambda\in\Aut(X)$ ($\lambda\in\Aut^-(X)$). Therefore, $\0\in\cP(X)$.
	\end{proof}
	
	%%\begin{remark}\label{3elements}
	%%If two maximal chains in $X$ have three or more elements in common, then a bijection $\0:B\to B$ cannot be increasing on one of them and decreasing on the other. Indeed, if $\0$ would be increasing on $C_1: u_1<u_2<\dots<u_m$ and decreasing on $C_2: v_1<v_2<\dots<v_{m'}$ with $x=u_i=v_{i'}$, $y=u_j=v_{j'}$, $z=u_k=v_{k'}$, then, by \cite[Lemma 5.5]{FKS}, 
	%%\begin{align*}
	%%\0(e_{xy})\0(e_{yz})= & \0(e_{u_iu_j})\0(e_{u_ju_k})=\0(e_{u_iu_k})=\0(e_{xz}),\\
	%%\0(e_{yz})\0(e_{xy})= & \0(e_{v_{j'}v_{k'}})\0(e_{v_{i'}v_{j'}})=\0(e_{v_{i'}v_{k'}})=\0(e_{xz}),
	%%\end{align*}
	%%which would imply $\0(e_{xy})\0(e_{yz})=\0(e_{yz})\0(e_{xy})$.
	%%\end{remark}
	
	\begin{lemma}\label{2elements}
		Let $\0\in\M(X)$. Let $C_1,C_2\in\C(X)$ such that $\0$ is increasing on $C_1$ and decreasing on $C_2$. If there exist $x,y\in C_1\cap C_2$ and $x<y$, then $x$ is the minimum of $C_1$ and $C_2$ and $y$ is the maximum of $C_1$ and $C_2$.
	\end{lemma}
	\begin{proof}
		We first notice that there is $z\in C_1$ such that $z<x$ ($y<z$) if, and only if, there is $z'\in C_2$ such that $z'<x$ ($y<z'$), by the maximality of $C_1$ 
		and $C_2$. Suppose that are $z\in C_1$ and $z'\in C_2$ such that $z,z'<x$. Then $\0(e_{zx})\0(e_{xy})=\0(e_{zy})$ and $\0(e_{xy})\0(e_{z'x})=\0(e_{z'y})$. Thus, there are $s<u<v<t$ such that $\0(e_{zx})=e_{su}$, $\0(e_{xy})=e_{uv}$, $\0(e_{z'x})=e_{vt}$ and $\0(e_{zy})=e_{sv}$, $\0(e_{z'y})=e_{ut}$. If $\0^{-1}$ is increasing on a maximal chain containing $s<u<v<t$, then $\0^{-1}(e_{st})=\0^{-1}(e_{su})\0^{-1}(e_{ut})=e_{zx}e_{z'y}$ which implies $z'=x$, a contradiction. If $\0^{-1}$ is decreasing on a maximal chain containing $s<u<v<t$, then $\0^{-1}(e_{st})=\0^{-1}(e_{vt})\0^{-1}(e_{sv})=e_{z'x}e_{zy}$ which implies $z=x$, a contradiction. Therefore, $x$ is the minimum of $C_1$ and $C_2$. Analogously, $y$ is the maximum of $C_1$ and $C_2$.
	\end{proof}
	
	\begin{lemma}\label{sigma_exists}
		For any $\0\in\M(X)$, there is $\sg:X^2_<\to K^*$ compatible with $\0$. 
	\end{lemma}
	\begin{proof}
		Let $\0\in\M(X)$ and $\mathcal C_i$ ($\mathcal C_d$) be the set of all maximal chains in $X$ on which $\0$ is increasing (decreasing). Let $(x,y)\in X^2_<$. If $x\in\Min(X)$, we set $\sigma(x,y)=1$. Otherwise, by \cref{2elements}, both $x$ and $y$ belong only to maximal chains from $\mathcal{C}_i$ or only to maximal chains from $\mathcal{C}_d$. In the former case we set $\sigma(x,y)=1$ and, in the latter one, we set $\sigma(x,y)=-1$.
		
		Let $x<y<z$ in $X$. Again, by \cref{2elements}, those three elements can be simultaneously only in maximal chains from $\mathcal{C}_i$ or only in maximal chains from $\mathcal{C}_d$. If they belong to maximal chains from $\mathcal{C}_i$, then $\sg(x,z)=1=\sg(x,y)\sg(y,z)$. Otherwise, there are two situations to be considered: $x\in\Min(X)$ or $x\not\in\Min(X)$. In the former case, $\sg(x,y)\sg(y,z)=1\cdot(-1)=-1=-\sg(x,z)$. In the latter case, $\sg(x,y)\sg(y,z)=(-1)^2=1=-\sg(x,z)$. Thus, $\sg$ is compatible with $\0$.  
	\end{proof}
	
	\begin{corollary}\label{AM(X)-group}
		The image of the group homomorphism $\wtl\laut(I(X,K))\to\M(X)$, $\vf\mapsto\0_\vf$, coincides with $\AM(X)$. In particular, $\AM(X)$ is a group.
	\end{corollary}
	\begin{proof}
		If $\vf\in\wtl\laut(I(X,K))$, then $\0_\vf\in\AM(X)$ by Lemma~5.4 and Remark~5.10 from \cite{FKS}. Let now $\0\in\AM(X)$. By \cref{sigma_exists} there is $\sg:X^2_<\to K^*$ compatible with $\0$. Choose an arbitrary $c=(c_1,\dots,c_n)\in K^n$ with $\sum_{i=1}^n c_i\in K^*$, where $n=|X|$. Then $\tau=\tau_{\0,\sg,c}\in \wtl\laut(I(X,K))$ such that $\0_\tau=\0$.
	\end{proof}
	
	\begin{theorem}\label{all-proper-iff-cP(X)=AM(X)}
		Every Lie automorphism of $I(X,K)$ is proper if and only if $\cP(X)=\AM(X)$.
	\end{theorem}
	\begin{proof}
		Suppose that every Lie automorphism of $I(X,K)$ is proper. Let $\0\in\AM(X)$. By \cref{sigma_exists}, there is $\sg:X^2_<\to K^*$ compatible with $\0$ and, by \cite[Lemma~5.16]{FKS}, there is $\vf\in \wtl\laut(I(X,K))$ inducing $(\0,\sg)$. By hypothesis, $\vf$ is proper. Therefore, $\0\in\cP(X)$ by \cref{theta_proper_vf_proper}.
		
		Conversely, suppose that $\cP(X)=\AM(X)$. Let $\vf\in \wtl\laut(I(X,K))$ inducing the pair $(\0,\sg)$. By hypothesis, $\0$ is proper. Therefore, $\vf$ is proper, by \cref{theta_proper_vf_proper}. Thus, every Lie automorphism of $I(X,K)$ is proper.
	\end{proof}
	
	\section{Admissible bijections of $B$ and maximal chains in $X$}\label{sec-max-chains}
	
	Observe that the definitions of the functions $s^{\pm}_{\0,\G}$ and $t^{\pm}_{\0,\G}$ make sense for any sequence $\G:u_0,\dots,u_m$ such that either $u_i<u_{i+1}$ or $u_{i+1}<u_i$ for all $0\le i\le m-1$. We will call such sequences $\G$ {\it semiwalks}. If moreover $u_0=u_m$, then $\G$ will be called a {\it closed semiwalk}.
	\begin{lemma}\label{from-G-to-G'-removement}
		Let $\0\in\M(X)$, $z\in X$ and $\G: u_0,u_1,\dots,u_m=u_0$ a closed semiwalk in $X$. Let $0\le k<k+l\le m$ such that $u_k<u_{k+1}<\dots<u_{k+l}$ or $u_k>u_{k+1}>\dots>u_{k+l}$ and set $\G':u_0,\dots,u_k,u_{k+l},\dots,u_m=u_0$. Then
		\begin{align}\label{s^+-_G-t^+-_G=s^+-_G'-t^+-_G'}
		s^+_{\0,\G}(z)-t^+_{\0,\G}(z)=s^+_{\0,\G'}(z)-t^+_{\0,\G'}(z),\ s^-_{\0,\G}(z)-t^-_{\0,\G}(z)=s^-_{\0,\G'}(z)-t^-_{\0,\G'}(z).
		\end{align}
	\end{lemma}
	\begin{proof}
		Assume that $u_k<u_{k+1}<\dots<u_{k+l}$. There are two cases.
		
		{\it Case 1.} $\0\m$ is increasing on a maximal chain containing $u_k<u_{k+1}<\dots<u_{k+l}$. Then there are $v_k<v_{k+1}<\dots<v_{k+l}$ such that $\0\m(e_{u_iu_j})=e_{v_iv_j}$ for all $k\le i<j\le k+l$.
		
		{\it Case 1.1.} $z=v_i$ for some $k<i<k+l$. If $\0(e_{zw})=e_{u_ju_{j+1}}$ for $w>z$ and $k\le j<k+l$, then $(z,w)=(v_j,v_{j+1})$, which implies that $j=i$ and $w=v_{i+1}$. Similarly $\0(e_{wz})=e_{u_ju_{j+1}}$ for $w<z$ and $k\le j<k+l$ yields $w=v_{i-1}$. Since, moreover, $\0(e_{v_kv_{k+l}})=e_{u_ku_{k+l}}$ and $z\ne v_k$, there is no $w>z$ such that $\0(e_{zw})=e_{u_ku_{k+l}}$. Similarly there is no $w<z$ such that $\0(e_{wz})=e_{u_ku_{k+l}}$. Therefore, $s^+_{\0,\G'}(z)=s^+_{\0,\G}(z)-1$, $t^+_{\0,\G'}(z)=t^+_{\0,\G}(z)-1$, $s^-_{\0,\G'}(z)=s^-_{\0,\G}(z)$ and $t^-_{\0,\G'}(z)=t^-_{\0,\G}(z)$.
		
		{\it Case 1.2.} $z=v_k$. Again, if $\0(e_{zw})=e_{u_ju_{j+1}}$ for $w>z$ and $k\le j<k+l$, then $w=v_{k+1}$. However, there is no $w<z$ such that $\0(e_{wz})=e_{u_ju_{j+1}}$ for some $k\le j<k+l$, but there is a unique $w>z$ (namely, $w=v_{k+l}$) such that $\0(e_{zw})=e_{u_ku_{k+l}}$. This means that $s^\pm_{\0,\G'}(z)=s^\pm_{\0,\G}(z)$ and $t^\pm_{\0,\G'}(z)=t^\pm_{\0,\G}(z)$.
		
		{\it Case 1.3.} $z=v_{k+l}$. This case is similar to Case 1.2. We have $s^\pm_{\0,\G'}(z)=s^\pm_{\0,\G}(z)$ and $t^\pm_{\0,\G'}(z)=t^\pm_{\0,\G}(z)$.
		
		{\it Case 1.4.} $z\not\in\{v_k,\dots,v_{k+l}\}$. Then there is neither $w<z$ such that $\0(e_{wz})=e_{u_ju_{j+1}}$ nor $w>z$ such that $\0(e_{zw})=e_{u_ju_{j+1}}$ for some $k\le j<k+l$. Moreover, there is neither $w<z$ such that $\0(e_{wz})=e_{u_ku_{k+l}}$ nor $w>z$ such that $\0(e_{zw})=e_{u_ku_{k+l}}$. Thus, $s^\pm_{\0,\G'}(z)=s^\pm_{\0,\G}(z)$ and $t^\pm_{\0,\G'}(z)=t^\pm_{\0,\G}(z)$.
		
		{\it Case 2.} $\0\m$ is decreasing on a maximal chain containing $u_k<u_{k+1}<\dots<u_{k+l}$. Then everything from Case 1 remains valid with the replacement of the ``$+$''-functions by their ``$-$''-analogs and vice versa.
		
		In any case $s^+_{\0,\G}(z)-t^+_{\0,\G}(z)$ and $s^-_{\0,\G}(z)-t^-_{\0,\G}(z)$ are invariant under the change of $\G$ for $\G'$. When $u_k>u_{k+1}>\dots>u_{k+l}$, the proof is analogous.
	\end{proof}
	
	\begin{corollary}\label{admissible-on-semiwalk}
		Let $\0\in\M(X)$. Then $\0\in\AM(X)$ if and only if \cref{s^+-s^-=t^+-t^-} holds
		for any $z\in X$ and any closed semiwalk $\G:u_0,\dots,u_m=u_0$, $m\ge 2$. 
	\end{corollary}
	\begin{proof}
		The ``if'' part is trivial. Let us prove the ``only if'' part. Indeed, the case $m=2$ is explained in the proof of \cite[Lemma~5.13]{FKS}, and if $m\ge 3$, then $\G$ can be extended to a closed walk $\Delta$ by inserting increasing (if $u_i<u_{i+1}$) or decreasing (if $u_i>u_{i+1}$) sequences of elements between $u_i$ and $u_{i+1}$ for all $0\le i\le m-1$. Since \cref{s^+-s^-=t^+-t^-} holds for $\Delta$, then by \cref{from-G-to-G'-removement} it holds for $\G$ too.
	\end{proof}
	
	\begin{lemma}\label{x_i=x_j-for-admissible-0}
		Let $\0\in\AM(X)$ and $\G:u_0,\dots,u_m=u_0$, $m\ge 2$, a closed semiwalk. Let also $x_i<y_i$, such that $\0(e_{u_iu_{i+1}})=e_{x_iy_i}$ for $u_i<u_{i+1}$ and $\0(e_{u_{i+1}u_i})=e_{x_iy_i}$ for $u_i>u_{i+1}$. 
		\begin{enumerate}
			\item If $x_i\in\Min(X)$ for some $0\le i\le m-1$, then there is $j\ne i$ such that $x_i=x_j$.\label{x_i-minimal}
			\item If $y_i\in\Max(X)$ for some $0\le i\le m-1$, then there is $j\ne i$ such that $y_i=y_j$.\label{y_i-maximal}
		\end{enumerate}
	\end{lemma}
	\begin{proof}
		We will prove \cref{x_i-minimal}, the proof of \cref{y_i-maximal} is analogous. Assume that $x_i\ne x_j$ for all $j\ne i$. If $u_i<u_{i+1}$, then $s^+_{\0\m,\G}(x_i)=1$ and $s^-_{\0\m,\G}(x_i)=0$, since $\0\m(e_{x_iy_i})=e_{u_iu_{i+1}}$ and $\0\m(e_{x_iw})\ne e_{u_ju_{j+1}}$ for any $w\ne y_i$ and $j\ne i$ (otherwise $x_i$ would coincide with some $x_j$ for $j\ne i$). Similarly, if $u_i>u_{i+1}$, then $s^+_{\0\m,\G}(x_i)=0$ and $s^-_{\0\m,\G}(x_i)=1$. Obviously, $t^{\pm}_{\0\m,\G}(x_i)=0$, because $x_i\in\Min(X)$. Thus, \cref{s^+-s^-=t^+-t^-} fails for the triple $(\0\m,\G,x_i)$, a contradiction.
	\end{proof}

	\begin{lemma}\label{intersect-of-2-max-chains-incr-and-decr}
		Let $\0\in\M(X)$. Assume that there exist $C,D\in\C(X)$ such that $\0$ is increasing on $C$ and decreasing on $D$. If $x\in C\cap D$, then either $x\in\Min(X)$ or $x\in\Max(X)$.
	\end{lemma}
	\begin{proof}
		Let $C:x_1<\dots<x_n$, $D:y_1<\dots<y_m$ and $x=x_i=y_j$ for some $1\le i\le n$ and $1\le j\le m$. Suppose that $1<i<n$. Then $1<j<m$, since otherwise $D$ would not be maximal. There exist maximal chains $C':u_1<\dots<u_n$ and $D':v_1<\dots<v_m$ such that $\0(e_{x_kx_l})=e_{u_ku_l}$ for all $1\le k<l\le n$ and $\0(e_{y_py_q})=e_{v_{m-q+1}v_{m-p+1}}$ for all $1\le p<q\le m$. In particular,
		$\0(e_{x_{i-1}x_i})=e_{u_{i-1}u_i}$, $\0(e_{x_ix_{i+1}})=e_{u_iu_{i+1}}$, $\0(e_{y_{j-1}y_j})=e_{v_{m-j+1}v_{m-j+2}}$, $\0(e_{y_jy_{j+1}})=e_{v_{m-j}v_{m-j+1}}$. Observe that $x_{i-1}<x_i=y_j<y_{j+1}$. Then either $u_i=v_{m-j}$, or $u_{i-1}=v_{m-j+1}$, depending on whether $\0$ is increasing or decreasing on a maximal chain containing $x_{i-1}<x_i<y_{j+1}$. Similarly, considering $y_{j-1}<y_j=x_i<x_{i+1}$ we obtain $v_{m-j+1}=u_{i+1}$ or $v_{m-j+2}=u_i$. If $u_i=v_{m-j}$, then $v_{m-j+1}=u_{i+1}$, so that $\{u_i,u_{i+1}\}\sst C'\cap D'$. However, $\0\m$ is increasing on $C'$ and decreasing on $D'$, so $u_i$ is the common minimum of $C'$ and $D'$ and $u_{i+1}$ is the common maximum of $C'$ and $D'$ by \cref{2elements}. This contradicts the assumption $1<i<n$. Similarly, $u_{i-1}=v_{m-j+1}$ implies $v_{m-j+2}=u_i$, whence $\{u_{i-1},u_i\}\sst C'\cap D'$ leading to a contradiction. 
		
		Thus, $i\in\{1,n\}$. If $i=1$, then necessarily $j=1$, as otherwise $C$ would not be maximal. Similarly, if $i=n$, then $j=m$.
	\end{proof}
	
	\begin{lemma}\label{intersect-of-2-max-chains-incr-and-incr}
		Let $\0\in\M(X)$ and $C,D\in\C(X)$, $C:x_1<\dots<x_n$, $D:y_1<\dots<y_m$. Assume that $x_i=y_j$ for some $1<i<n$ and $1<j<m$. If $\0$ is increasing (resp. decreasing) on $C$, then it is increasing (resp. decreasing) on $D$. Moreover, if $\0(C):u_1<\dots<u_n$ and $\0(D):v_1<\dots<v_m$, then $u_i=v_j$ (resp. $u_{n-i+1}=v_{m-j+1}$).
	\end{lemma}
	\begin{proof}
		Let $\0$ be increasing on $C$. Then it is increasing on $D$ by \cref{intersect-of-2-max-chains-incr-and-decr}. Using the same idea as in the proof of \cref{intersect-of-2-max-chains-incr-and-decr}, we have $\0(e_{x_{i-1}x_i})=e_{u_{i-1}u_i}$, $\0(e_{x_ix_{i+1}})=e_{u_iu_{i+1}}$, $\0(e_{y_{j-1}y_j})=e_{v_{j-1}v_j}$, $\0(e_{y_jy_{j+1}})=e_{v_jv_{j+1}}$. Considering $x_{i-1}<x_i=y_j<y_{j+1}$ we conclude that $u_i=v_j$ or $u_{i-1}=v_{j+1}$. Similarly it follows from $y_{j-1}<y_j=x_i<x_{i+1}$ that $u_i=v_j$ or $v_{j-1}=u_{i+1}$. If $u_i\ne v_j$, then $u_{i-1}=v_{j+1}$ and $v_{j-1}=u_{i+1}$. But this is impossible, since $u_{i-1}<u_{i+1}$ and $v_{j+1}>v_{j-1}$.
		
		The proof for the decreasing case is analogous.
	\end{proof}
	
	% Similarly one proves the following.
	% \begin{remark}\label{intersect-of-2-max-chains-decr-and-decr}
	%     Let $\0:B\to B$ be a bijection monotone on maximal chains in $X$. Assume that there exist maximal chains $C:x_1<\dots<x_n$ and $D:y_1<\dots<y_m$ such that $\0$ is decreasing on both $C$ and $D$. Let $C':u_1<\dots<u_n$ and $D':v_1<\dots<v_m$ be the images of $C$ and $D$ under $\0$. If $x_i=y_j$ for some $1<i<n$ and $1<j<m$, then $u_{n-i+1}=v_{m-j+1}$.
	% \end{remark}
	
	\begin{definition}
		Let $C,D\in\C(X)$. We say that $C$ and $D$ are {\it linked} if there exists $x\in C\cap D$ such that $x\not\in\Min(X)\sqcup\Max(X)$. Denote by $\sim$ the equivalence relation on $\C(X)$ generated by $\{(C,D)\in\C(X)^2 : C,D\text{ are linked}\}$.
	\end{definition}
	
	\begin{lemma}\label{0-is-inc-or-decr-on-all-equiv-chains}
		Each $\0\in\M(X)$ induces a bijection $\wtl\0$ on $\C(X)/{\sim}$. Moreover, if $\0$ is increasing (resp. decreasing) on $C\in\C(X)$, then it is increasing (resp. decreasing) on any $D\sim C$.
	\end{lemma}
	\begin{proof}
		Let $C,D\in\C(X)$ be linked. Then $\0(C)$ and $\0(D)$ are linked by \cref{intersect-of-2-max-chains-incr-and-incr}. It follows that $C\sim D$ implies $\0(C)\sim\0(D)$, which induces a map $\wtl\0:\C(X)/{\sim}\to\C(X)/{\sim}$. It is a bijection whose inverse is $\wtl{\0\m}$.
		
		Assume that $\0$ is increasing on $C\in\C(X)$. Then by \cref{intersect-of-2-max-chains-incr-and-incr} it is increasing on any $D\in\C(X)$ which is linked to $C$. By the obvious induction this extends to any $D\sim C$. The decreasing case is similar.
	\end{proof}
	
	\begin{definition}
		Given $\fC\in\C(X)/{\sim}$, we define the {\it support} of $\fC$, denoted $\supp(\fC)$, as the set $\{x\in C : C\in\fC\}$.
	\end{definition}
	
	\begin{remark}
		Let $\fC,\fD\in\C(X)/{\sim}$. If $\fC\ne\fD$, then $\supp(\fC)\cap\supp(\fD)\sst\Min(X)\sqcup\Max(X)$. 
		
		Indeed, assume that $x\in \supp(\fC)\cap\supp(\fD)$, where $x\not\in\Min(X)$ and $x\not\in\Max(X)$. There are $C\in\fC$ and $D\in\fD$ such that $x\in C\cap D$. But then $C$ and $D$ are linked, so $C\sim D$, whence $\fC=\fD$.
	\end{remark}
	
	\begin{theorem}\label{supp(C)-and-supp(0(C))}
		Let $\0\in\AM(X)$ and $\fC\in\C(X)/{\sim}$. Then there exists an isomorphism or an anti-isomorphism of posets $\lb:\supp(\fC)\to\supp(\wtl\0(\fC))$ such that for all $x<y$ from $\supp(\fC)$ one has
		\begin{align}\label{0=hat-lambda-on-supp(C)}
		\0(e_{xy})=\hat\lb(e_{xy}).
		\end{align}
	\end{theorem}
	\begin{proof}
		In view of \cref{0-is-inc-or-decr-on-all-equiv-chains} we may assume that $\0$ is increasing on all $C\in\fC$ or decreasing on all $C\in\fC$. Consider the case of an increasing $\0$. We are going to construct the corresponding $\lb:\supp(\fC)\to\supp(\wtl\0(\fC))$. Let $x\in \supp(\fC)$ and $C:x_1<\dots<x_n$ a maximal chain from $\fC$ containing $x$. Denote by $C':u_1<\dots<u_n$ the image of $C$ under $\0$. If $x=x_i$ for some $1\le i\le n$, then we put $\lb(x_i)=u_i$. We still need to show that the definition does not depend on the choice of $C$. 
		
		If $1<i<n$, then this is true by \cref{intersect-of-2-max-chains-incr-and-incr}. 
		
		If $i=1$, then $x\in\Min(X)$. If there exists another maximal chain $D:y_1<\dots<y_m$ from $\fC$ containing $x$, then $x=y_1$. We thus need to show that $u_1=v_1$, where $D':v_1<\dots<v_m$ is the image of $D$ under $\0$. Since $C\sim D$, there are $C=C_1,\dots,C_k=D$ such that $C_j$ and $C_{j+1}$ are linked for all $1\le j\le k-1$. Denote by $z_j$ an element of $C_j\cap C_{j+1}$, $1\le j\le k-1$, which is neither minimal, nor maximal in $X$. Set also $z_0=z_k=x$. Observe that $z_j,z_{j+1}\in C_{j+1}$ for all $0\le j\le k-1$, so that either $z_j\le z_{j+1}$ or $z_j\ge z_{j+1}$. Let $\G:z_0,z_1,\dots,z_k=z_0$. Clearly, $z_0\ne z_j$ and $z_j\ne z_k$ for all $1\le j\le k-1$, as $z_0=z_k\in\Min(X)$, while $z_j\not\in\Min(X)$. Moreover, we will assume that $z_j\ne z_{j+1}$ for all $1\le j\le k-2$, since otherwise we may just remove the repetitions (and at least $2$ elements will remain). Let also $a_j<b_j$ such that $\0(e_{z_jz_{j+1}})=e_{a_jb_j}$ if $z_j<z_{j+1}$, and $\0(e_{z_{j+1}z_j})=e_{a_jb_j}$ if $z_j>z_{j+1}$, $0\le j\le k-1$. Observe that $a_0=u_1$, since $x=z_0<z_1\in C$, and $a_{k-1}=v_1$, since $x=z_k<z_{k-1}\in D$. In particular, $a_0,a_{k-1}\in\Min(X)$. Since $z_j$ is not minimal for all $1\le j\le k-2$, then neither is $a_j$, so that $a_j\not\in\{a_0,a_{k-1}\}$ for such $j$. But then we must have $a_0=a_{k-1}$, i.e., $u_1=v_1$, by \cref{x_i=x_j-for-admissible-0}\cref{x_i-minimal}.
		
		The case $i=n$ is similar. The map $\lb:\supp(\fC)\to\supp(\wtl\0(\fC))$ is thus constructed.
		
		We now prove that $\lb(x)<\lb(y)$ and \cref{0=hat-lambda-on-supp(C)} holds for all $x<y$ from $\supp(\fC)$. By construction, this is true for $x$ and $y$ belonging to the same $C\in\fC$. Let now $x<y$ be arbitrary elements of $\supp(\fC)$. Choose $C\in\fC$ containing $x$ and $D\in\fC$ containing $y$. If $x\not\in\Min(X)$, then any $C'\in\C(X)$ containing $x$ and $y$ is linked to $C$, so that $C'\in\fC$. The case when $y\not\in\Max(X)$ is similar. Let now $x\in\Min(X)$ and $y\in\Max(X)$. As above, we choose $C=C_1,\dots,C_k=D$, such that $C_j,C_{j+1}\in\fC$ are linked for all $1\le j\le k-1$, and $z_j\in C_j\cap C_{j+1}$, $1\le j\le k-1$, which is neither minimal, nor maximal in $X$. We set $z_0=x$, $z_k=y$ and $\G:z_0,z_1,\dots,z_k,z_{k+1}=z_0$. We also assume that $z_j\ne z_{j+1}$ for all $0\le j\le k$ and denote by $a_j<b_j$ the elements satisfying $\0(e_{z_jz_{j+1}})=e_{a_jb_j}$ if $z_j<z_{j+1}$ and $\0(e_{z_{j+1}z_j})=e_{a_jb_j}$ if $z_j>z_{j+1}$, $0\le j\le k$. As above, observe that $a_0,a_k\in\Min(X)$, while $a_1,\dots,a_{k-1}\not\in\Min(X)$. Then $a_0=a_k$ by \cref{x_i=x_j-for-admissible-0}\cref{x_i-minimal}. Similarly it follows from \cref{x_i=x_j-for-admissible-0}\cref{y_i-maximal} that $b_{k-1}=b_k$. But $a_0=\lb(x)$ and $b_{k-1}=\lb(y)$, since $e_{a_0b_0}=\0(e_{z_0z_1})=e_{\lb(z_0)\lb(z_1)}=e_{\lb(x)\lb(z_1)}$ and $e_{a_{k-1}b_{k-1}}=\0(e_{z_{k-1}z_k})=e_{\lb(z_{k-1})\lb(z_k)}=e_{\lb(z_{k-1})\lb(y)}$. Hence, $\lb(x)=a_k<b_k=\lb(y)$ and $\0(e_{xy})=\0(e_{z_{k+1}z_k})=e_{a_kb_k}=e_{\lb(x)\lb(y)}=\hat\lb(e_{xy})$.
		
		It is clear that $\lb$ is a bijection whose inverse is the map $\mu:\supp(\wtl\0(\fC))\to\supp(\fC)$ corresponding to $\0\m$. Thus, $\lb$ is an isomorphism between $\supp(\fC)$ and $\supp(\wtl\0(\fC))$.
		
		The case of a decreasing $\0$ is analogous.
	\end{proof}
	
	The following example shows that the admissibility of $\0$ in \cref{supp(C)-and-supp(0(C))} cannot be dropped.
	\begin{example}
		Let $X=\{1,\dots,10,1',\dots 9',7''\}$ with the following Hasse diagram.
		
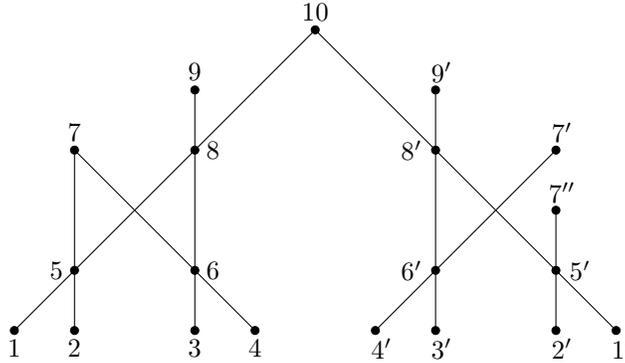
\begin{center}
			\begin{tikzpicture}[line cap=round,line join=round,>=triangle 45,x=0.8cm,y=0.8cm]
			\draw  (-4,0)-- (-4,1);
			\draw  (-4,1)-- (-4,3);
			\draw  (-5,0)-- (-4,1);
			\draw  (-4,1)-- (-2,3);
			\draw  (-2,3)-- (-2,4);
			\draw  (-4,3)-- (-2,1);
			\draw  (-2,1)-- (-1,0);
			\draw  (-2,0)-- (-2,1);
			\draw  (-2,1)-- (-2,3);
			\draw  (-2,3)-- (0,5);
			\draw  (2,4)-- (2,3);
			\draw  (2,3)-- (2,1);
			\draw  (2,1)-- (2,0);
			\draw  (1,0)-- (2,1);
			\draw  (2,1)-- (4,3);
			\draw  (2,3)-- (4,1);
			\draw  (4,1)-- (5,0);
			\draw  (4,1)-- (4,0);
			\draw  (4,2)-- (4,1);
			\draw  (0,5)-- (2,3);
			\draw [fill=black] (-1,0) circle (1.5pt);
			\draw[color=black] (-1,-0.3) node {$4$};
			\draw [fill=black] (-2,1) circle (1.5pt);
			\draw[color=black] (-1.7,1) node {$6$};
			\draw [fill=black] (-2,0) circle (1.5pt);
			\draw[color=black] (-2,-0.3) node {$3$};
			\draw [fill=black] (-2,3) circle (1.5pt);
			\draw[color=black] (-1.7,3) node {$8$};
			\draw [fill=black] (-2,4) circle (1.5pt);
			\draw[color=black] (-2,4.3) node {$9$};
			\draw [fill=black] (-4,3) circle (1.5pt);
			\draw[color=black] (-4,3.3) node {$7$};
			\draw [fill=black] (-4,1) circle (1.5pt);
			\draw[color=black] (-4.3,1) node {$5$};
			\draw [fill=black] (-4,0) circle (1.5pt);
			\draw[color=black] (-4,-0.3) node {$2$};
			\draw [fill=black] (-5,0) circle (1.5pt);
			\draw[color=black] (-5,-0.3) node {$1$};
			\draw [fill=black] (0,5) circle (1.5pt);
			\draw[color=black] (0,5.3) node {$10$};
			\draw [fill=black] (2,3) circle (1.5pt);
			\draw[color=black] (1.6,3) node {$8'$};
			\draw [fill=black] (2,4) circle (1.5pt);
			\draw[color=black] (2.1,4.3) node {$9'$};
			\draw [fill=black] (2,1) circle (1.5pt);
			\draw[color=black] (1.6,1) node {$6'$};
			\draw [fill=black] (2,0) circle (1.5pt);
			\draw[color=black] (2.1,-0.3) node {$3'$};
			\draw [fill=black] (1,0) circle (1.5pt);
			\draw[color=black] (1.1,-0.3) node {$4'$};
			\draw [fill=black] (4,1) circle (1.5pt);
			\draw[color=black] (4.4,1) node {$5'$};
			\draw [fill=black] (4,3) circle (1.5pt);
			\draw[color=black] (4.1,3.3) node {$7'$};
			\draw [fill=black] (5,0) circle (1.5pt);
			\draw[color=black] (5.1,-0.3) node {$1'$};
			\draw [fill=black] (4,0) circle (1.5pt);
			\draw[color=black] (4.1,-0.3) node {$2'$};
			\draw [fill=black] (4,2) circle (1.5pt);
			\draw[color=black] (4.1,2.3) node {$7''$};
			\end{tikzpicture}
		\end{center}
		Then $\C(X)/{\sim}$ consists of $2$ classes whose supports are $Y=\{1,\dots,10\}$ and $Z=\{1',\dots,9',10,7''\}$. Observe that there exists $\0\in\M(X)$ mapping one ${\sim}$-class to another. It is defined as follows: $\0(e_{ij})=e_{i'j'}$ for all $i\le j$ in $X$ with $(i,j)\ne (5,7)$, $\0(e_{i'j'})=e_{ij}$ for all $i'\le j'$ in $X$ with $(i',j')\ne(5',7'')$, $\0(e_{57})=e_{5'7''}$ and $\0(e_{5'7''})=e_{57}$ (to make the definition shorter, we set $10':=10$). However, $Y$ and $Z$ are not isomorphic or anti-isomorphic because $|Y|\ne|Z|$. The reason is that $\0\not\in\AM(X)$. Indeed, for $\G:5<7>6<8>5$ we have $s^\pm_{\0,\G}(7')=0$, $t^+_{\0,\G}(7')=0$ and $t^-_{\0,\G}(7')=1$.
	\end{example}
	
	As a consequence of \cref{supp(C)-and-supp(0(C)),all-proper-iff-cP(X)=AM(X)} we have the following result which generalizes \cite[Corollary 5.19]{FKS}, where $X$ was a chain.
	\begin{corollary}\label{|C(X)-over-tilde|=1}
		If $|\C(X)/{\sim}|=1$, then each $\vf\in\laut(I(X,K))$ is proper.
	\end{corollary}

	Observe, however, that the condition $|\C(X)/{\sim}|=1$ is not necessary for all $\vf\in\laut(I(X,K))$ to be proper, as the following example shows.
	
	\begin{example}
		Let $X=\{1,2,3,4,5,6\}$ with the following Hasse diagram.
		\begin{center}
			% \begin{tikzpicture}[line cap=round,line join=round,>=triangle 45,x=0.8cm,y=0.8cm]
			% \draw (-2,4)-- (-1,2);
			% \draw (-1,2)-- (0,0);
			% \draw (0,0)-- (1,2);
			% \draw (1,2)-- (2,4);
			% \draw (-1,2)-- (0,4);
			% \draw (0,4)-- (1,2);
			% \draw [fill=black] (-1,2) circle (1.5pt);
			% \draw (-1.3,2) node {2};
			% \draw [fill=black] (1,2) circle (1.5pt);
			% \draw (1.3,2) node {3};
			% \draw [fill=black] (-2,4) circle (1.5pt);
			% \draw (-2,4.3) node {4};
			% \draw [fill=black] (2,4) circle (1.5pt);
			% \draw (2,4.3) node {6};
			% \draw [fill=black] (0,0) circle (1.5pt);
			% \draw (0,-0.3) node {1};
			% \draw [fill=black] (0,4) circle (1.5pt);
			% \draw (0,4.3) node {5};
			% \end{tikzpicture}
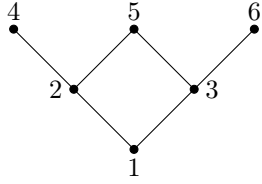
			\begin{tikzpicture}[line cap=round,line join=round,>=triangle 45,x=0.8cm,y=0.8cm]
			\draw (-2,2)-- (-1,1);
			\draw (-1,1)-- (0,0);
			\draw (0,0)-- (1,1);
			\draw (1,1)-- (2,2);
			\draw (-1,1)-- (0,2);
			\draw (0,2)-- (1,1);
			\draw [fill=black] (-1,1) circle (1.5pt);
			\draw (-1.3,1) node {2};
			\draw [fill=black] (1,1) circle (1.5pt);
			\draw (1.3,1) node {3};
			\draw [fill=black] (-2,2) circle (1.5pt);
			\draw (-2,2.3) node {4};
			\draw [fill=black] (2,2) circle (1.5pt);
			\draw (2,2.3) node {6};
			\draw [fill=black] (0,0) circle (1.5pt);
			\draw (0,-0.3) node {1};
			\draw [fill=black] (0,2) circle (1.5pt);
			\draw (0,2.3) node {5};
			\end{tikzpicture}
		\end{center}
		Note that $\C(X)/{\sim}$ consists of $2$ classes whose supports are $Y=\{1,2,4,5\}$ and $Z=\{1,3,5,6\}$. For any $\0\in\AM(X)$ there are $2$ possibilities for the corresponding isomorphisms $\lb_1$ and $\lb_2$ between the supports: either $\lb_1:Y\to Y$ and $\lb_2:Z\to Z$, or $\lb_1:Y\to Z$ and $\lb_2:Z\to Y$. In the former case $\lb_1=\id_Y$ and $\lb_2=\id_Z$, and in the letter case $\lb_2=\lb\m_1$, where $\lb_1$ maps an element $y\in Y$ to the element $z\in Z$ which is symmetric to $y$ with respect to the vertical line passing through the vertices $1$ and $5$. In both cases $\lb_1$ and $\lb_2$ are the restrictions of an automorphism of $X$ to $Y$ and $Z$, respectively.
	\end{example}
	
	\section{Sets of length one}\label{sec-length-one}
	
	\subsection{Admissible bijections of $B$ and crowns in $X$}
	Before proceeding to the case $l(X)=1$ we will prove a useful fact which holds for $X$ of an arbitrary length.
	\begin{definition}
		Let $n$ be an integer greater than $1$. By a {\it weak $n$-crown} we mean a poset $P=\{x_1,\dots,x_{n},y_1,\dots,y_{n}\}$ where 
		\begin{align}\label{order-in-a-crown}
		x_i<y_i\mbox{ for all }1\le i\le n,\ x_{i+1}<y_i\mbox{ for all }1\le i\le n-1\mbox{ and }x_1<y_{n}.
		\end{align}
		An {\it $n$-crown} is a weak $n$-crown which has no other pairs of distinct comparable elements except \cref{order-in-a-crown}. It is thus fully determined by $n$ up to an isomorphism and will be denoted by $\Cr_n$. A poset $P$ is called a {\it weak crown (resp. crown)}, if it is a weak $n$-crown (resp. $n$-crown) for some $n\ge 2$. We say that a poset $X$ {\it has a weak crown (resp. crown)} if there is a subset $Y\subseteq X$ which is a weak crown (resp. crown) under the induced partial order.
	\end{definition}
	Posets without crowns are known to satisfy some ``good'' properties~\cite{Rival76,Draxler94,DN2016}.
	% \begin{remark}
	%     Let $X$ be a weak crown and $B=\{e_{xy}\mid x<y\}$. Then there exists a non-admissible bijection of $B$. Indeed, if $X=\{x_1,\dots,x_k,y_1,\dots,y_k\}$ with \cref{order-in-a-crown}, then we define $\0(e_{x_1y_1})=e_{x_1y_k}$, $\0(e_{x_1y_k})=e_{x_1y_1}$ and $\0(e_{uv})=\0(e_{uv})$ for any other pair of elements $u<v$ of $X$.
	% \end{remark}
	
	\begin{lemma}\label{admissibility-on-crowns}
		Let $\0\in\M(X)$. Then $\0\in\AM(X)$ if and only if \cref{s^+-s^-=t^+-t^-} holds
		for any $z\in X$ and any weak crown $\G:u_0,\dots,u_m=u_0$ in $X$.
	\end{lemma}
	\begin{proof}
		The ``only if'' part is obvious. For the ``if'' part take any closed semiwalk $\G:u_0,u_1,\dots,u_m=u_0$ and $z\in X$. Let $0\le k<k+l\le m$ such that $u_k<u_{k+1}<\dots<u_{k+l}$. We define $\G':u_0,\dots,u_k,u_{k+l},\dots,u_m=u_0$. By \cref{admissible-on-semiwalk} equality \cref{s^+-s^-=t^+-t^-} holds for $\G$ if and only if it holds for $\G'$. The same is true for any $\G'$ obtained from $\G$ by removing intermediate terms in a decreasing sequence of consecutive vertices. Thus, doing so for all maximal sequences in $\G$ we finally get $\G'$ whose vertices form either a sequence $x<y>x$, or a weak crown. However, the case $\G':x<y>x$ can be ignored, because \cref{s^+-s^-=t^+-t^-} always holds for such $\G'$ as shown in the proof of \cite[Lemma~5.13]{FKS}.
	\end{proof}
	
	\subsection{The crownless case}
	Let now $l(X)=1$. Observe that $\M(X)=S(B)$. Moreover, any $C\in\C(X)$ is linked only to itself, so $|\C(X)/{\sim}|=|\C(X)|$ and \cref{supp(C)-and-supp(0(C))} becomes useless. 
	\begin{definition}
		We say that $\0\in\M(X)$ {\it is separating} if there exists a pair of non-disjoint $C,D\in\C(X)$ such that $\0(C)$ and $\0(D)$ are disjoint.
	\end{definition}
	
	\begin{remark}
		Any separating $\0$ is not proper.
	\end{remark}
	
	\begin{lemma}\label{exist-disjoint-C-and-D}
		Let $l(X)=1$. If $|\Min(X)|>1$ and $|\Max(X)|>1$, then there are disjoint $C,D\in\C(X)$.
	\end{lemma}
	\begin{proof}
		% Let $x_0$ a fixed element of $Min(X)$. Then either there exists $y_1\in Max(X)$ which is not comparable with $x_0$ or $x_0\leq y$ for all $y\in Max(X)$.
		
		% In the first case, since $X$ is connected, there exists $x_1\in Min(X)$ such that $x_1\leq y_1$. For the same reason, there exists $y_0\in Max(X)$ such that $x_0\leq y_0$ and $C=\{x_0,y_0\}$ and $D=\{x_1,y_1\}$ are disjoint.
		
		% In the second case, let $x_1\in Min(X)-\{x_0\}$. Again by the connectness of $X$, there exists $y_1\in Max(X)$ such that $x_1\leq y_1$. Since $|\max(X)\|\geq 2$, there exists $y_0\in Max(X)-\{y_1\}$ and $x_0\leq y_0$, since $x_0\leq y$ for all $y\in Max(X)$ and, again, $C=\{x_0,y_0\}$ and $D=\{x_1,y_1\}$ are disjoint.
		
		Choose arbitrary $x<y$ in $X$. Obviously, $x\in\Min(X)$ and $y\in\Max(X)$. Let $U=\{u\in \Min(X)\mid u\not\le y\}$ and $V=\{v\in \Max(X)\mid x\not\le v\}$. If $U\ne\emptyset$, then take $u\in U$. Clearly, $u\ne x$. Since $X$ is connected, there exists $v>u$, and $v\ne y$ by the definition of $U$. Then $C:x<y$ and $D:u<v$ are disjoint. The case $V\ne\emptyset$ is similar. Suppose now that $U=V=\emptyset$. This means that $x\le v$ for any $v\in\Max(X)$ and $y\ge u$ for any $u\in\Min(X)$. Choose $u\in\Min(X)\setminus\{x\}$ and $v\in\Max(X)\setminus\{y\}$. Then $C_1:x<v$ and $D_1:u<y$ are disjoint.
	\end{proof}
	
	\begin{proposition}\label{M(X)=P(X)-when-Min(X)=1-or-Max(X)=1}
		Let $l(X)=1$. Then $\M(X)=\cP(X)$ if and only if $|\Min(X)|=1$ or $|\Max(X)|=1$.
	\end{proposition}
	\begin{proof}
		If $|\Min(X)|=1$, say $\Min(X)=\{x\}$, then any $\0\in\M(X)$ can be identified with a bijection $\lb$ of $\Max(X)$ such that $\0(e_{xy})=e_{x\lb(y)}$ for all $y>x$. But $\lb$ extends to an automorphism of $X$ by means of $\lb(x)=x$, so that $\0(e_{xy})=e_{\lb(x)\lb(y)}$. A symmetric argument works in the case $|\Max(X)|=1$.
		
		Suppose now that $|\Min(X)|>1$ and $|\Max(X)|>1$. By \cref{exist-disjoint-C-and-D} there are disjoint $C:x<y$ and $D:u<v$. Choose a path $x=x_0,x_1,\dots,x_m=u$. Since $x,u\in\Min(X)$, then $m\ge 2$ and $x_0<x_1>x_2$. We define $\0(e_{x_0x_1})=e_{xy}$, $\0(e_{xy})=e_{x_0x_1}$, $\0(e_{x_2x_1})=e_{uv}$, $\0(e_{uv})=e_{x_2x_1}$ and $\0(e_{ab})=e_{ab}$ for any other $e_{ab}\in B$. Clearly, $\0\in S(B)=\M(X)$ and it is separating, in particular, not proper.
	\end{proof}
	
	If $X$ does not contain a weak crown, then $\AM(X)=\M(X)$ by \cref{admissibility-on-crowns}. Hence, we obtain the following.
	\begin{corollary}\label{P(X)=AM(X)-for-X-without-crown}
		Let $l(X)=1$ and assume that $X$ does not contain a weak crown. Then $\AM(X)=\cP(X)$ if and only if $|\Min(X)|=1$ or $|\Max(X)|=1$.
	\end{corollary}
	
	\subsection{The crown case}
	We will now consider two classes of posets of length one which have crowns. We begin with the case of $X$ being a crown and are going to calculate the groups $\cP(X)$ and $\AM(X)$ explicitly. Thus, in this subsection $X=\Cr_n=\{x_1,\dots,x_n,y_1,\dots,y_n\}$.
	
	\begin{definition}
		The chains $x_i<y_i$, $1\le i\le n$, will be called {\it odd}, and $x_{i+1}<y_i$, $1\le i\le n-1$, and $x_1<y_{n}$ will be called {\it even}. Thus, each element of $\Cr_n$ belongs to exactly one odd chain and to exactly one even chain.
	\end{definition}
	
	\begin{lemma}\label{0(C)-and-0(D)-have-opposite-parities}
		Let $\0\in\M(\Cr_n)$. Then $\0\in\AM(\Cr_n)$ if and only if for any pair of distinct non-disjoint chains $C,D\in\C(\Cr_n)$ the images $\0(C)$ and $\0(D)$ have opposite parities.
	\end{lemma}
	\begin{proof}
		Observe that \cref{s^+-s^-=t^+-t^-} is invariant under cyclic shifts of $\G$ (any such shift does not change the functions $s^{\pm}_{\0,\G}$ and $t^{\pm}_{\0,\G}$). Thus, for admissibility it is enough to consider $\G:x_1<y_1>x_2<\dots <y_{n}>x_1$, since any other cycle in $\Cr_n$ is a cyclic shift of $\G$. 
		
		The {\it ``only if''} case. Let $\0\in\AM(\Cr_n)$ and $C,D\in\C(\Cr_n)$, such that $C\cap D=\{z\}$, where $z\in\Min(\Cr_n)\sqcup\Max(\Cr_n)$. Suppose that $z\in \Min(\Cr_n)$. Then there are only two elements $w,w'\in\Max(\Cr_n)$ such that $z<w,w'$. Thus, $t_{\0,\G}^{\pm}(z)=0$ and $\0$ is admissible if and only if $s_{\0,\G}^{+}(z)=s_{\0,\G}^{-}(z)=1$, which only occur if $\0(C)$ and $\0(D)$ have opposite parities. The case when $z\in\Max(\Cr_n)$ is similar.
		
		The {\it ``if''} case. Let $\0\in\M(\Cr_n)$ and $z\in \Cr_n$ be arbitrary. Again, we consider the case $z\in\Min(\Cr_n)$, so that $t_{\0,\G}^{\pm}(z)=0$. Choose $w,w'\in\Max(\Cr_n)$ with $z<w,w'$ and put $C:z<w$ and $D:z<w'$. Since $\0(C)$ and $\0(D)$ have opposite parities, then $s_{\0,\G}^{\pm}(z)=1$ and \cref{s^+-s^-=t^+-t^-} is satisfied. Similarly, one handles the case $z\in\Max(\Cr_n)$.
		%  such that $C\cap D=\{x_i\}$. Thus, either $C:x_i<y_i$ and $D:x_i<y_{i-1}$, $2\le i\le k$, or $C:x_1<y_1$ and $D:x_1<y_k$. We consider the first situation, since the second one can be obtained from the first one identifying $y_0$ with $y_k$. If $\0(C)$ and $\0(D)$ are odd, then $s^+_{\0,\G}(x_i)=2$, $s^-_{\0,\G}(x_i)=0$ and $t^{\pm}_{\0,\G}(x_i)=0$, a contradiction. Similarly, if $\0(C)$ and $\0(D)$ are even, then $s^+_{\0,\G}(x_i)=0$, $s^-_{\0,\G}(x_i)=2$ and $t^{\pm}_{\0,\G}(x_i)=0$, a contradiction. The case, when $C\cap D=\{y_i\}$ is analogous (we have $s^{\pm}_{\0,\G}(y_i)=0$ and $(t^+_{\0,\G}(y_i),t^-_{\0,\G}(y_i))\in\{(0,2),(2,0)\}$).
	\end{proof}
	
	% \begin{lemma}
	%  Let $X$ be a $2$-crown. Then each $\0\in\AM(X)$ is proper.
	% \end{lemma}
	% \begin{proof}
	%  Let $\0\in\AM(X)$. Denote $C_1:x_1<y_1$, $C_2:x_2<y_1$, $C_3:x_2<y_2$ and $C_4:x_1<y_2$. 
	%  Then $\0(C_1)$ and $\0(C_2)$ have opposite parities by \cref{0(C)-and-0(D)-have-opposite-parities}. Having fixed $(\0(C_1),\0(C_2))$, due to \cref{0(C)-and-0(D)-have-opposite-parities} there is only one possibility to choose $(\0(C_3),\0(C_4))$, namely, $\0(C_3)$ is the chain whose parity is opposite to that of $\0(C_2)$ and $\0(C_4)$ is the chain whose parity is opposite to that of $\0(C_1)$. Thus, any $\0\in\AM(X)$ is uniquely determined by $(\0(C_1),\0(C_2))$. There are $4$ options to choose $\0(C_1)$, and for each of them there are $2$ options to choose $\0(C_2)$. Thus, $|\AM(X)|\le 8$. On the other hand, there are $4$ automorphisms of $X$ and $4$ anti-automorphisms of $X$, each of them inducing a proper $\0$ which is automatically admissible. Hence, $|\AM(X)|=8$ and each $\0\in\AM(X)$ is proper.
	% \end{proof}
	
	\begin{proposition}\label{AM(C_n)-as-semidirect-product}
		The group $\AM(\Cr_n)$ is isomorphic to $(S_{n}\times S_{n})\rtimes \mathds Z_2$.
	\end{proposition}	
	\begin{proof}
		Denote by $\mathcal{O}$ and $\mathcal{E}$ the subsets of $\C(\Cr_n)$ formed by the odd and even chains, respectively, and let $\mathcal{G}=\{\0\in \M(\Cr_n) : \0(\mathcal{O})=\mathcal{O} \text{ or } \0(\mathcal{O})=\mathcal{E}\}$. We will first prove that $\AM(\Cr_n)=\mathcal{G}$. For any $\0\in \AM(\Cr_n)$, if $\0(e_{x_1y_1})\in \mathcal{O}$, then $\0(e_{x_2y_1})\in \mathcal{E}$ by \cref{0(C)-and-0(D)-have-opposite-parities}. It follows that $\0(e_{x_2y_2})\in \mathcal{O}$ by the same reason. Applying this argument consecutively to $e_{x_3y_2},e_{x_3y_3},\dots,e_{x_ky_k},e_{x_1y_k}$, we obtain $\0(\mathcal{O})=\mathcal{O}$. Similarly, if $\0(e_{x_1y_1})\in \mathcal{E}$, then $\0(\mathcal{O})=\mathcal{E}$. Thus, $\0\in \mathcal{G}$. On the other hand, let $\0\in \mathcal{G}$ and $C_1,C_2\in\C(\Cr_n)$, $C_1\neq C_2$, such that $C_1\cap C_2\neq\emptyset$. Then $C_1$ and $C_2$ have opposite parities, say, $C_1\in \mathcal{O}$ and $C_2\in \mathcal{E}$. If $\0(\mathcal{O})=\mathcal{O}$, then $\0(\mathcal{E})=\mathcal{E}$ due to the bijectivity of $\0$. Analogously, if $\0(\mathcal{O})=\mathcal{E}$, then $\0(\mathcal{E})=\mathcal{O}$. So, in either case, $\0(C_1)$ and $\0(C_2)$ have opposite parities. Therefore, $\0\in \AM(\Cr_n)$, by \cref{0(C)-and-0(D)-have-opposite-parities}.
		
		We now prove that $\mathcal{G}\cong (S_{n}\times S_{n})\rtimes \mathds Z_2$. Consider $\mathcal{H}=\{\0\in \M(\Cr_n) : \0(\mathcal{O})=\mathcal{O}\}$. Clearly, $\mathcal{H}$ is a (normal) subgroup of $\mathcal{G}$ of index $2$. Since $|\mathcal{O}|=|\mathcal{E}|=n$, we have $\mathcal{H}\cong S_{n}\times S_{n}$. Define $\0\in\M(\Cr_n)$ as follows: $\0(e_{x_iy_i})=e_{x_{i+1}y_i}$ and $\0(e_{x_{i+1}y_i})=e_{x_iy_i}$ for $1\le i\le n-1$, $\0(e_{x_{n}y_{n}})=e_{x_1y_{n}}$ and $\0(e_{x_1y_{n}})=e_{x_{n}y_{n}}$. By definition $\0\in\mathcal{G}\setminus\mathcal{H}$ and $\0$ has order $2$. Therefore, $\mathcal{G}=\mathcal{H}\cdot\langle\0\rangle\cong(S_{n}\times S_{n})\rtimes \mathds Z_2$.
	\end{proof}	
	
	% \begin{corollary}
	% \red{We have} $|\AM(\Cr_n)|=2(\red{n}!)^2$.
	% \end{corollary}
	
	% \begin{example}
	%  Let $X$ be a $k$-crown for $k\ge 3$. Then there exists $\0\in\AM(X)$ which is not proper.
	
	%  Indeed, if $X=\{x_1,\dots,x_k,y_1,\dots,y_k\}$, then we define $\0(e_{x_1y_1})=e_{x_2y_2}$, $\0(e_{x_2y_2})=e_{x_1y_1}$ and $\0(e_{uv})=e_{uv}$ for $(u,v)\not\in\{(x_1,y_1),(x_2,y_2)\}$. To prove that $\0$ is admissible, it suffices to consider the pairs of non-disjoint chains involving $x_1<y_1$ or $x_2<y_2$. There are $4$ such pairs. We have $\0(e_{x_1y_1})=e_{x_2y_2}$ and $\0(e_{x_1y_k})=e_{x_1y_k}$, $\0(e_{x_1y_1})=e_{x_2y_2}$ and $\0(e_{x_2y_1})=e_{x_2y_1}$, $\0(e_{x_2y_2})=e_{x_1y_1}$ and $\0(e_{x_2y_1})=e_{x_2y_1}$, $\0(e_{x_2y_2})=e_{x_1y_1}$ and $\0(e_{x_3y_2})=e_{x_3y_2}$. In any case, the images have opposite parities. Thus, $\0\in\AM(X)$. However, $\0$ is not proper, because it is separating. In fact, $\0(e_{x_1y_1})=e_{x_2y_2}$ and $\0(e_{x_1y_k})=e_{x_1y_k}$, where $k\ne 2$.
	% \end{example}
	
	\begin{proposition}\label{P(C_n)-as-semidirect-product}
		The group $\cP(\Cr_n)$ is isomorphic to $\mathds Z_{2n}\rtimes \mathds Z_2$.
	\end{proposition}
	\begin{proof}
		In view of \cref{P(X)-cong-Aut^pm(X)} it suffices to prove that $\Aut^\pm(\Cr_n)\cong \mathds Z_{2n}\rtimes \mathds Z_2$. To this end, we will show that $\Aut^\pm(\Cr_n)\cong D_{2n}$, where $D_{2n}$ is the group  of symmetries of a regular $2n$-gon which is known to be isomorphic to $\mathds Z_{2n}\rtimes \mathds Z_2$. Denote $x_i$ by $u_{2i-1}$ and $y_i$ by $u_{2i}$, for all $i=1,\ldots,n$, identifying $u_j$ with the $j$-th vertex of a regular $2n$-gon, whose vertices are indexed consecutively according to the counterclockwise orientation. For the sake of simplicity, we shall consider the indices modulo $2n$ in the rest of the proof. 
		
		Given $\varphi\in\Aut^\pm(\Cr_n)$, set $i_{\varphi}$ to be the integer modulo $2n$ such that $\varphi(u_{2n})=u_{i_{\varphi}}$. Notice that if $i_{\varphi}$ is even then $\varphi\in\Aut(\Cr_n)$, otherwise $\varphi\in\Aut^-(\Cr_n)$. In any case, since the only elements of $\Cr_n$ comparable with $u_{2n}$, besides itself, are $u_{2n-1}$ and $u_1$, then $\varphi(u_1)=u_{i_{\varphi}\pm 1}$. If $\varphi(u_1)=u_{i_{\varphi}+1}$, then it can be easily shown inductively that $\varphi(u_j)=u_{i_{\varphi}+j}$ for any $j=1,\ldots,2n$. This corresponds to the counterclockwise rotation by an angle of $i_{\varphi}\pi/n$ in $D_{2n}$. If $\varphi(u_1)=u_{i_{\varphi}-1}$, again by an easy inductive argument, $\varphi(u_j)=u_{i_{\varphi}-j}$ for all $j=1,\ldots,2n$. If $i_{\varphi}$ is even, $\varphi$ corresponds to the reflection across the diagonal containing $u_j$ and $u_{j+n}$, where $2j=i_{\varphi}$. Otherwise $\varphi$ corresponds to the reflection across the line which contains the midpoints of the sides $u_ju_{j+1}$ and $u_{j+n}u_{j+n+1}$, where $i_{\varphi}=2j+1$. Since $i_{\varphi}$ can be any of the $2n$ indices of the vertices considered, all the $4n$ elements of $D_{2n}$ ($2n$ rotations and $2n$ reflections) can occur as elements of $\Aut^\pm(\Cr_n)$ and we obtain the claimed isomorphism.
	\end{proof}	
	
	% \begin{corollary}
	%  Let $X$ be a $k$-crown. Then $|\cP(X)|=4k$.
	% \end{corollary}
	
	\begin{corollary}\label{P(X)-and-AM(X)-for-X=2-crown}
		We have $\cP(\Cr_2)=\AM(\Cr_2)$ and $\cP(\Cr_n)\ne \AM(\Cr_n)$ for all $n>2$.
	\end{corollary}
	\begin{proof}
		Indeed, $|\AM(\Cr_2)|=|\cP(\Cr_2)|$ and $|\AM(\Cr_n)|=2(n!)^2>4n!>4n=|\cP(\Cr_2)|$ for $n>2$.
	\end{proof}
	
	\subsection{The case of the ordinal sum of two anti-chains}
	
	We will now proceed to the case of sets of length one which have as many crowns as possible.
	\begin{definition}
		Given positive integers $m$ and $n$, denote by $\K_{m,n}$ the poset $\{x_1,\dots,x_m,y_1,\dots,y_n\}$, where $x_i<y_j$ for all $1\le i\le m$ and $1\le j\le n$, and there is no other pair of distinct comparable elements.
	\end{definition}
	Observe that $\K_{m,n}$ is the ordinal sum~\cite{Stanley} of two anti-chains of cardinalities $m$ and $n$. The Hasse diagram of $\K_{m,n}$ is a {\it complete bipartite graph}~\cite{Bollobas}, so that $\Aut(\K_{m,n})\cong S_m\times S_n$. It is also clear that $\K_{m,n}$ is anti-isomorphic to $\K_{n,m}$, so we may assume that $m\le n$. The cases $m=1$ and $m=n=2$ (a $2$-crown) were considered in \cref{M(X)=P(X)-when-Min(X)=1-or-Max(X)=1,P(X)-and-AM(X)-for-X=2-crown}.

	\begin{proposition}\label{P(X)=AM(X)-for-X=K_mn}
		Let $2\le m\le n$. Then $\cP(\K_{m,n})=\AM(\K_{m,n})$.
	\end{proposition}
	\begin{proof}
		Let $\0\in\AM(\K_{m,n})$. Fix $j\in \{1,\ldots, n\}$ and write $\0(e_{x_iy_j})=e_{u_iv_i}$, $1\le i\le m$. Denote by $U_j$ and $V_j$ the sets of all $u_i$ and $v_k$, respectively. We first prove that for any pair of $u_i\in U_j$ and $v_k\in V_j$ there is $l$ such that $\0(e_{x_ly_j})=e_{u_iv_k}$. This is trivial if $u_i=u_k$ or $v_i=v_k$, so let $u_i\ne u_k$ and $v_i\ne v_k$. Consider the cycle $\G:u_i<v_i>u_k<v_k>u_i$. We have $s^{\pm}_{\0,\G}(y_j)=0$ and $t^+_{\0,\G}(y_j)=2$. Since $\0$ is admissible, we must have $t^-_{\0,\G}(y_j)=2$. But this means that $\0(e_{x_ly_j})=e_{u_iv_k}$ for some $1\le l\le m$ (and $\0(e_{x_py_j})=e_{u_kv_i}$ for some $1\le p\le m$), as desired. As a consequence, we obtain $|U_j|\cdot|V_j|=m$.
		
		We now prove that $|U_j|=1$ or $|V_j|=1$. Assume that $|U_j|\ge 2$ and $|V_j|\ge 2$. Since $|U_j|\cdot|V_j|=m$, we conclude that $|U_j|\le \frac m2$ and $|V_j|\le \frac m2\le\frac n2$. It follows that there exist $z<w$ such that $z\not\in U_j$ and $w\not\in V_j$. Consider the cycle $\G:u_1<v_1>z<w>u_1$. Clearly, $s^{\pm}_{\0,\G}(y_j)=0$, $t^+_{\0,\G}(y_j)=1$ and $t^-_{\0,\G}(y_j)=0$, a contradiction.
		
		\textit{Case 1.} $|V_j|=1$ for all $1\le j\le n$. Then there exists a bijection $\lb$ of $\{y_1,\dots,y_n\}$ such that $\{\0(e_{x_1y_j}),\dots,\0(e_{x_my_j})\}=\{e_{x_1\lb(y_j)},\dots,e_{x_m\lb(y_j)}\}$ for all $1\le j\le n$. We will prove that $\0(e_{x_iy_j})=e_{u_i\lb(y_j)}$ and $\0(e_{x_iy_k})=e_{z_i\lb(y_k)}$ imply $u_i=z_i$ for $j\ne k$. Suppose that $u_i\ne z_i$ and consider the cycle $\G:u_i<\lb(y_j)>z_i<\lb(y_k)>u_i$. We have $t^{\pm}_{\0,\G}(x_i)=0$, $s^+_{\0,\G}(x_i)=2$ and $s^-_{\0,\G}(x_i)=0$, a contradiction. Thus, there exists a bijection $\mu$ of $\{x_1,\dots,x_m\}$ such that $\0(e_{x_iy_j})=e_{\mu(x_i)\lb(y_j)}$ for all $1\le i\le m$ and $1\le j\le n$. But this means that $\0$ corresponds to the automorphism of $\K_{m,n}$ acting as $\lb$ on $\Max(\K_{m,n})$ and as $\mu$ on $\Min(\K_{m,n})$. So, $\0\in\cP(\K_{m,n})$.
		
		\textit{Case 2.} $|U_j|=1$ for some $1\le j\le n$. Then $|V_j|=m$. We will prove that this is possible only if $m=n$. Assume that $m<n$ and let $U_j=\{p_j\}$. Since $\0(e_{x_iy_j})=e_{p_jv_i}$ for all $1\le i\le m$ and $|V_j|<n$, there exist $z<w$ such that $z\ne p_j$ and $w\not\in V_j$. Taking the cycle $\G:p_j<v_1>z<w>p_j$, we obtain $s^{\pm}_{\0,\G}(y_j)=0$, $t^+_{\0,\G}(y_j)=1$ and $t^-_{\0,\G}(y_j)=0$, a contradiction. Thus, $m=n$. We now prove that $|U_k|=1$ {\it for all} $1\le k\le n$. If $|U_k|\ne 1$, then $k\ne j$ and $|V_k|=1$, say $V_k=\{q_k\}$. We have $\{\0(e_{x_1y_j}),\dots,\0(e_{x_ny_j})\}=\{e_{p_jy_1},\dots,e_{p_jy_n}\}$ and $\{\0(e_{x_1y_k}),\dots,\0(e_{x_ny_k})\}=\{e_{x_1q_k},\dots,e_{x_nq_k}\}$. Since $j\ne k$, these sets must be disjoint. But $e_{p_jq_k}$ belongs to their intersection, a contradiction. Thus, $|U_k|=1$ for all $1\le k\le n$. Replacing $\0$ by $\0'\circ\0$, where $\0'(e_{xy})=e_{\mu(y)\mu(x)}$ and $\mu$ is the anti-automorphism of $X$ which interchanges $x_i$ and $y_i$ for all $1\le i\le n$, we get the situation of Case 1, so that $\0'\circ\0\in\cP(X)$. Since $\0'\in\cP(\K_{m,n})$, we conclude that $\0\in\cP(\K_{m,n})$.
	\end{proof}
	
	\section*{Acknowledgements}
	The second author was partially supported by CNPq 404649/2018-1.
	
	%  Notice that $\0(C_1)$ and $\0(C_2)$ determine uniquely an admissible $\0\in \AM(X)$. Indeed, there are only two even maximal chains two odd ones. Since $\0(C_1)$ and $\0(C_2)$ must have opposite parities due to $C_1\cap C_2\neq\varnothing$, it remains only one option for $\0(C_4)$, which also must have the opposite parity of $\0(C_1)$ because $C_1\cap C_4\neq\varnothing$. Once $\0(C_1),\0(C_2),\0(C_4)$ are determined, so is $\0(C_3)$. 
	
	% Let us count $|\AM(X)|$. There are four options for $\0(C_1)$. Fixed $\0(C_1)$, it remains two options for $\0(C_2)$ since it has opposite parity of $\0(C_1)$. As we have just seen, this determines uniquely $\0$. By the multiplicative principle, there are eight possibilities for $\0$. Since, by the construction of the possibilities for $\0$, we have $\0(C_1)$ and $\0(C_3)$ with the same parity and $\0(C_2)$ and $\0(C_4)$ with the same parity too (opposite to that of $\0(C_1)$), all those possibilities are admissible and $|\AM(X)|=8$.
	
	\bibliography{bibl}{}
	\bibliographystyle{acm}

\end{document}